\documentclass[11pt,twoside,english,reqno,a4paper]{amsart}

\usepackage{listings,graphicx,amsmath,varioref,amscd,amssymb,color,bm,stmaryrd,amsthm,amsfonts,graphics,geometry,latexsym,pgf,pst-all} 
\theoremstyle{plain}
\usepackage{esint}
\usepackage{amsthm}

\usepackage{enumerate}

\theoremstyle{plain}
\newtheorem{theorem}{Theorem}[section]
\newtheorem{proposition}[theorem]{Proposition}

\newtheorem{lemma}[theorem]{Lemma}

\usepackage{geometry}
\usepackage{cancel}
\usepackage[colorlinks=true]{hyperref}

\theoremstyle{definition}
\newtheorem{defin}[theorem]{Definition}

\newtheorem{remark}[theorem]{Remark}

\theoremstyle{remark}

\usepackage{fouriernc}
\usepackage[T1]{fontenc}

\numberwithin{equation}{section}

\def\supp{\text{\text{supp}}}

\DeclareMathOperator{\R}{\mathbb{R}}
\DeclareMathOperator{\N}{\mathbb{N}}

\newcommand{\car}[1]{\raise1pt\hbox{$\chi$}_{#1}}

\begin{document}
\title[Singular anisotropic problems]{Existence and uniqueness of weak solutions to singular anisotropic elliptic problems}

\author[F. Esposito]{Francesco Esposito}
\author[F. Oliva]{Francescantonio Oliva}
\author[E. Vecchi]{Eugenio Vecchi}

\address[Francesco Esposito]{Dipartimento di Matematica e Informatica, Universit\`a della Calabria
\hfill \break\indent
Ponte Pietro Bucci Cubo 31B, 87036 Rende (CS), Italy}
\email{\tt francesco.esposito@unical.it}

\address[Francescantonio Oliva]{
Dipartimento di Scienze di Base e Applicate per l' Ingegneria, Sapienza Universit\`a di Roma
\hfill \break\indent
Via Scarpa 16, 00161 Roma, Italy}
\email{\tt francescantonio.oliva@uniroma1.it}

\address[Eugenio Vecchi]{Dipartimento di Matematica, Universit\`a di Bologna
\hfill \break\indent
Piazza di Porta San Donato 5, 40126 Bologna, Italy}
\email{\tt eugenio.vecchi2@unibo.it}

\makeatletter
\@namedef{subjclassname@2020}{%
  \textup{2020} Mathematics Subject Classification}
\makeatother

\begin{abstract}
In this paper we consider a quasilinear singular anisotropic elliptic problem with a $p$-sublinear perturbation. We prove uniqueness of weak solutions and we provide necessary and sufficient conditions for the existence of weak solutions in the same spirit of the celebrated paper by Lazer and McKenna.
\end{abstract}


\thanks{All the authors are member of INdAM and this work is part of the GNAMPA Project 2025 {\em Problemi singolari e degeneri: esistenza, unicita' e analisi delle proprieta' qualitative delle soluzioni} - CUP $E5324001950001$. F.E. is partially supported by the PID Project (Spain) PID2024-155314NB-I00 and by PRIN Project (Italy) P2022YFAJH {\em Linear and Nonlinear PDE's: New directions and Applications}. E.V. is partially supported by PRIN Project (Italy) 2022R537CS {\em ``$NO^3$ - Nodal Optimization, NOnlinear elliptic equations, NOnlocal geometric problems, with a focus on regularity''.}}

\maketitle

\section{Introduction}

In this paper we deal with existence and uniqueness of weak solutions to a quasilinear anisotropic elliptic problem which involves a singular and a $p$-sublinear nonlinearity
\begin{equation}
	\label{PbMain}
	\begin{cases}
	\displaystyle -\Delta_p^H u=\frac{f(x)}{u^\gamma} + h(x)u^\theta & \text { in } \Omega, \\
	u>0 & \text { in } \Omega, \\
	u=0 & \text { on } \partial \Omega,
	\end{cases}
\end{equation}
where $\Omega \subset \mathbb{R}^N$ ($N\ge 2$) is an open bounded set with smooth boundary and
\begin{equation}\label{eq:def_Operator}
\Delta_p^H u :=\operatorname{div}\left(H^{p-1}(\nabla u) \nabla H(\nabla u)\right)
\end{equation}
is the so-called anisotropic Finsler $p$-laplacian, with $p>1$ and $H:\mathbb{R}^{N}\to [0,+\infty)$ is a function which introduces possible anisotropies. Finally, $\gamma>0$, $0<\theta <p-1$, $f$ is a nonnegative function (not identically zero) which can be merely integrable in $\Omega$ and $h$ is a bounded and non-negative function on $\Omega$.
Clearly, if $H(\xi)=|\xi|$ the operator $\Delta_p^H(\cdot)$ reduces to the classical (isotropic) $p$-laplacian. 

In the applications it often happens to face problems which naturally exhibits anisotropic behaviours; this is for example the case of diffusions in heterogeneous materials or crystalline structures (see e.g. \cite{CaHo,Gu,DeMaMiNe}) and anisotropic curvature flows (see e.g. \cite{BePa, DePGX,Cabezas}), but it appears in image processing as well, (see e.g. \cite{PeMa,EsOs}). From a purely geometrical point of view, a quite natural setting allowing for anisotropies is the one of Finsler geometry (see e.g. \cite{Bao}): roughly speaking, differently from the Riemannian case, one chooses a norm instead of a inner product to measure objects living on the tangent bundle. Once this choice is made, it is therefore possible to properly define an analogous of the Laplace-Beltrami operator as in \eqref{eq:def_Operator}. \\
Our perspective here is to require $H$ to satisfy \eqref{Hh1}, \eqref{Hh2} and \eqref{Hh3} listed below in Section \ref{sec:Preliminaries}. 
This approach is closer to the recent studies of PDEs driven by $\Delta_p^H(\cdot)$. A partial list of contributions could mention regularity results (see e.g. \cite{CFV14,CSR,AnCiFa,AnCiCiFaMa}), overdetermined problems (see e.g. \cite{CiSal,WangXia, Bi-Ci-Sa,BiCi,CiSal2}) and more general nonlinear problems (see e.g. \cite{CFV16,CFR,FaLi1, DiPoVa,FaScVu,EMSV,BERV}).
Due to our interest in \eqref{PbMain}, let us finally mention a few results dealing with singular problems (see e.g. \cite{DePG2,EST,MonSciuTro}).

\medskip

Let us now go back to the model problem \eqref{PbMain}. We are interested in the following issues:
\begin{itemize}
	\item[i)] uniqueness of weak solutions, i.e. $u \in W^{1,p}_{0}(\Omega)$ which weakly solves \eqref{PbMain};
	\item[ii)] existence of weak solutions to \eqref{PbMain} with $0 \leq h \in L^{\infty}(\Omega)$ and under different assumptions on $f$.
\end{itemize}

We postpone to Section \ref{sec:Uniqueness} (see Definition \ref{defweak}) and Section \ref{sec:Existence} (see Definition \ref{defweakmod}) for the precise notion of weak solution used in i) and ii) respectively.

As long as i) is concerned, uniqueness of weak solutions has to be expected due to the strictly decreasing behaviour of the nonlinear part as well described e.g. in \cite{Brezis_Oswald}. For this reason, and only for what concerns uniqueness, we are therefore able to deal with the following more general problem:
\begin{equation}\label{intro:Pb_section_notion}
	\begin{cases}
		\displaystyle -\Delta_p^H u= F(x,u) & \text { in } \Omega, \\
		u=0 & \text { on } \partial \Omega.
	\end{cases}
\end{equation}
where $1<p<N$ and the function $F:\Omega\times (0,+\infty)\to [0,+\infty)$ is a Carath\'eodory function satisfying the following assumptions
\begin{itemize}
	\item the map $s \mapsto F(x,s)$ is continuous for almost every $x \in \Omega$;
	\item the function $x \mapsto F(x,s)$ is measurable for every fixed $s \in \mathbb{R}$.
\end{itemize}
\noindent Using a test function argument as in \cite{DurOl}, which in turn is inspired by the approach developed in \cite{Brezis_Oswald}, we are able to prove our first result, which reads as follows:
\begin{theorem}\label{intro:thm:uniqueness}
	Let $H$ satisfy \eqref{Hh1}-\eqref{Hh3} and let $1<p<N$. Moreover let $F:\Omega \times (0,+\infty) \to [0,+\infty)$ be a Carath\'{e}odory function satisfying
	\begin{equation}\label{intro:eq:Ipotesi_su_F}
		\dfrac{F(x,s)}{s^{p-1}} \quad \textrm{ is strictly decreasing w.r.t. } s \textrm{ for a.e. } x \in \Omega.
	\end{equation}
	Then, there exists at most one weak solution $u \in W^{1,p}_0(\Omega)$ to \eqref{intro:Pb_section_notion}.
\end{theorem}
\medskip
%
%

The situation is more delicate once we consider the problem of the existence of weak solutions as mentioned in ii). From now on, we will deal only with the model problem \eqref{PbMain}. We recall that the non-negative coefficient $h$ in front of the $p$-sublinear term will always assumed to be bounded, except for Theorem \ref{intro:thm:existencegammaminoredi1}. We refer to e.g. \cite{BoOr} for the case of unbounded $h$. With a nod to the nowadays classical Lazer-McKenna type results \cite{lazer}, we allow for an additional perturbative $p$-sublinear term but we keep our focus on the roles of both $\gamma$ and the summability of $f$. Just to give a preliminary flavor of the phenomena we are interested in, let us briefly recall the famous result in \cite{lazer}: consider the purely singular problem
\begin{equation*}
	\begin{cases}
	\displaystyle -\Delta u= f(x) u^{-\gamma} & \text { in } \Omega, \\
	u>0 & \text { in } \Omega,\\
	u=0 & \text { on } \partial \Omega,
\end{cases}
\end{equation*}
\noindent with $0< c \leq f \in C^{\alpha}(\overline{\Omega})$. Lazer and McKenna proved that there exists a unique classical solution which is also a weak solution if and only if $\gamma <3$. We refer to e.g. \cite{Sun,sz} for extensions of this result with more general $f$ and also with the addition of a sublinear term as we are considering. See also \cite{BoOr2, OPsurvey} for a more general discussion.

In this perspective, the first result related to ii) considers the case of $f \in L^{\infty}(\Omega)$ and it holds for all $\gamma>0$.

\begin{theorem}	\label{intro:thm:existence1}
	Let $H$ satisfy \eqref{Hh1}-\eqref{Hh3}, let $1<p<N$, let $0<c \le f\in L^\infty(\Omega)$ with $\gamma>0$, $0 \leq h \in L^\infty(\Omega)$, $0\le \theta<p-1$. Then there exists a solution $u \in W^{1,p}_0(\Omega)$ to \eqref{PbMain} if and only if $\gamma < 2 + \frac{1}{p-1}$.
\end{theorem}  

The second result holds even for an unbounded $f$ but with $\gamma >1$. It looks once again as a necessary and sufficient condition on $\gamma$, but this time as to be interpreted as follows: there exists a function $f\in L^m(\Omega)$ for which no solution in $W^{1,p}_0(\Omega)$ is expected. The precise statement reads as follow:

\begin{theorem}	\label{intro:thm:existence2}
	Let $H$ satisfy \eqref{Hh1}-\eqref{Hh3}, let $1<p<N$, $0 \leq h \in L^\infty(\Omega)$, $0\le \theta<p-1$ and let $\gamma>1$. Then there exists a solution $u \in W^{1,p}_0(\Omega)$ for every  $0< f \in L^m(\Omega)$ ($m>1$) to \eqref{PbMain} if and only if $\gamma < 2 + \frac{1}{p-1} - \frac{p}{(p-1)m}$.
\end{theorem}  
\begin{remark}
	Let us explicitly stress that, as $m \to +\infty$, the previous theorem gives the threshold $\gamma < 2 + \frac{1}{p-1}$ found in Theorem \ref{intro:thm:existence1} and which coincides with the one of \cite{lazer} in case of the Laplacian operator. 
\end{remark} 

\medskip

Let us now provide a brief description of the techniques involved and of the related results needed along the path to prove both Theorem \ref{intro:thm:existence1} and Theorem \ref{intro:thm:existence2}.

\begin{itemize}
	\item We consider first the case of $0<\gamma \leq 1$ and $0 \leq f \in L^{(\frac{p^*}{1-\gamma})'}(\Omega)$ with $\left(\frac{p^*}{1-\gamma}\right)' = 1$ if $\gamma =1$ and we show the existence of a weak solution as limit of solutions of an auxiliary truncated problem \eqref{pbn}. This is the content of Theorem \ref{intro:thm:existencegammaminoredi1}. We stress that this immediately provides the existence of a solution in Theorem \ref{intro:thm:existence1} in the range $0<\gamma \leq 1$.
	\item We provide optimal controls from above and below for every {\it bounded} weak solution to \eqref{PbMain} in terms of suitable powers of the first eigenfunction. The proof is based on a sub-super-solution method and uses the regularity up to the boundary (see Remark \ref{rmk:regularity}). This is the content of Theorem \ref{thm:estbelowabove} which partially improves \cite{MonSciuTro}.
	\item For $\gamma >1$ and $0<f\in L^{1}(\Omega)$ we provide a necessary and sufficient condition (sometimes called {\it compatibility condition} in the literature) for the existence of a weak solution to \eqref{PbMain}. This consists in requiring the existence of a function $u_{0} \in W^{1,p}_{0}(\Omega)$ such that 
	\begin{equation*}
		\int_{\Omega}f |u_0|^{1-\gamma} < +\infty.
	\end{equation*}
	This is the content of Theorem \ref{thmexistence1} and it is mainly based on \cite{Sun,sz}. We stress that the existence part of this result is based on solving a constrained minimization problem over a suitable set which is not the classical Nehari manifold, and this is due to the singular part of the nonlinearity. We also stress that along the proof we need an anisotropic version of the classical {\it convergence of the gradients} (see e.g. \cite{BoMu92}) which we prove in Proposition \ref{prop:convgrad}.
	\item We prove Theorem \ref{intro:thm:existence1} as follows: for the existence part (in the range $1<\gamma<2+\tfrac{1}{p+1}$) we exploit both Theorem \ref{thmexistence1} (choosing $u_0$ as a suitable power of the first eigenfunction) and the estimates in Theorem \ref{thm:estbelowabove}. The optimality of the upper bound for $\gamma$ is proved once again exploiting Theorem \ref{thm:estbelowabove} and actually shows that there cannot exist a weak solution to \eqref{PbMain} for $\gamma \geq 2+\tfrac{1}{p+1}$.
	\item The proof of Theorem \ref{intro:thm:existence2} follows the same path as before as long as the use of Theorem \ref{thmexistence1} and Theorem \ref{thm:estbelowabove} is concerned: indeed we choose once again $u_0$ to be a suitable power of the first eigenfunction. As already mentioned before, the optimality has to be understood in a slightly different way: we exploit a construction already performed in the case of the laplacian (i.e. $H(\xi)=|\xi|, p=2, h=0$), see e.g. \cite[Theorem 5.7]{OPsurvey}, which shows the existence of a function $f$ for which no weak solution exists with  $\gamma \geq 2+\tfrac{1}{p+1}$.
\end{itemize}
\medskip
Let us spend a final comment on \eqref{Hh3} satisfied by $H$. As explained in Remark \ref{rmk:regularity}, this is the crucial assumption that permits to inherit the classical boundary regularity for the $p$-laplacian. This is at the moment unavoidable for us since we make also use of the Hopf Lemma proved in \cite{CSR}.

\medskip

The paper is organized as follows: in Section \ref{sec:Preliminaries} we introduce the notation and we collect all the preliminary results needed in what follows, included Proposition \ref{prop:convgrad}. In Section \ref{sec:Uniqueness} we first provide the definition of solution to \eqref{intro:Pb_section_notion} and then we prove Theorem \ref{intro:thm:uniqueness}. In Section \ref{sec:Existence} we give the precise notion of solution we are considering for \eqref{PbMain} and we prove Theorem \ref{thmexistence1}, 

\section{Notation and preliminaries}\label{sec:Preliminaries}

In the entire paper $\Omega \subset \mathbb{R}^{N}$ is an open and bounded set with smooth enough boundary $\partial \Omega$. For any set $E$, $|E|$ denotes its the $N$-dimensional Lebesgue measure.
For $r \in (1,+\infty)$, we denote by $r':= \tfrac{r}{r-1}$ the H\"{o}lder conjugate exponent of $r$.

Given any function $v$, we denote the positive and negative part of $v$ by 
\begin{equation*}\label{eq:def_positive_negative_part}
	v^{+}:= \max\{v,0\} \quad \textrm{ and } \quad v^{-}:= -\min\{v,0\},
\end{equation*}
and, for every fixed $k>0$, the truncation function $T_k:\mathbb{R}\to \mathbb{R}$ defined as
\begin{equation*}\label{eq:def_Tk}
	T_{k}(s):= \max \{ -k, \min\{s,k\}\},\quad s \in \mathbb{R}.
\end{equation*}	
As $\Omega$ is smooth, in order to avoid technicalities, we denote by $d(x)$ a suitable positive smooth regularization of $\mathrm{dist}(x,\partial\Omega)$ which agrees with it in a neighbourhood of $\partial\Omega$.
 Let us also stress that, along the paper, $C$ (with subscript in some cases) denotes a generic numerical constant allowed to vary from line to line and whose dependence on other quantities is not relevant.

For the sake of simplicity, if no ambiguity is possible, we use the notation:
$$\int_\Omega f:= \int_\Omega f(x) \ dx.$$

\smallskip 

Now we set some properties and geometrical tools about the anisotropic elliptic operator which are fundamental for the rest of the paper.

\smallskip

In the entire work $H: \R^N \rightarrow \R$ is a function satisfying:
\begin{equation}\label{Hh1}
H \in C_{\text {loc }}^2(\mathbb{R}^N \backslash\{0\}) \text{ is even and such that $H(\xi)>0$ for all $\xi \in \mathbb{R}^N \backslash\{0\}$},
\end{equation}
\begin{equation}\label{Hh2}
H(t \xi)=t H(\xi) \text{ for all } \xi \in \mathbb{R}^N \backslash\{0\}, \text{ for all } t>0,
\end{equation}
and 
\begin{equation}\label{Hh3}
\exists \lambda>0 \text{ such that } D^2 H(\xi) \zeta \cdot \zeta \geq \lambda|\zeta|^2, \text{ for all } \xi \in \partial \mathcal{B}_1^H \text{ and for all } \zeta \in \nabla H(\xi)^{\perp},
\end{equation}
where $\mathcal{B}_1^H:=\left\{\xi \in \mathbb{R}^N: H(\xi)<1\right\}$.

Let us explicitly stress that \eqref{Hh3} means that $H$ is uniformly elliptic, i.e.~$\mathcal B_1^H$ is uniformly convex.

Now we recall some useful properties of the Finsler norm, for which we refer to  \cite{CFV14, CFV16}.
First observe that, if $H: \R^N \rightarrow \R$ satisfies \eqref{Hh1}-\eqref{Hh3}, then one can show that
	\begin{equation}\label{controlloH}
	\exists \alpha, \beta >0: \alpha |\xi|\leq H(\xi)\leq \beta |\xi|,\quad\forall\, \xi\in\R^N;
\end{equation}
Moreover, as $H$ is $1$-homogeneous, it follows from the Euler Theorem that
\begin{equation}\label{eq:eulero}
	\nabla H(\xi) \cdot \xi = H(\xi), \qquad\forall\, \xi \in \R^N\setminus \{0\}.
\end{equation}
 Furthermore $\nabla H$ is $0$-homogeneous so that
\begin{equation}\label{eq:grad0omog}
	\nabla H(t\xi)=\hbox{sign}(t)\nabla  H(\xi), \qquad\forall\,\xi \in \R^N\setminus \{0\},\, \forall t\neq 0.
\end{equation}
Let also stress that, using \eqref{eq:grad0omog}, there exists a positive constant $K>0$ such that 
\begin{equation}\label{eq:gradHlimitato}
	|\nabla H(\xi)| \le  K, \qquad \forall \xi\in\R^N\setminus \{0\}.
\end{equation}

Now, we recall some vectorial inequalities, whose proofs can be found in \cite{CSR, EMSV}.
\begin{proposition} \label{prop:Lindqvist} 
	Let $H$ satisfy \eqref{Hh1}-\eqref{Hh3}. Then, for any $p>1$ and  $\eta, \eta' \in \R^N$ such that $|\eta|+|\eta'|>0$, it holds
	\begin{equation}\label{eq:boundbelow}
		[H^{p-1}(\eta)\nabla H(\eta)-H^{p-1}(\eta ')\nabla H(\eta ')] \cdot (\eta-\eta') \ge \tilde C_p(|\eta|+|\eta'|)^{p-2}|\eta-\eta'|^2,
	\end{equation} 
	and
	\begin{equation}\label{eq:boundabove}
		|H^{p-1}(\eta)\nabla H(\eta)-H^{p-1}(\eta ')\nabla H(\eta ')|\le \tilde C_p(|\eta|+|\eta'|)^{p-2}|\eta-\eta'|.
	\end{equation}
	Moreover, for any $p \geq 2$ it holds the following inequality
	\begin{equation} \label{introeq:ineqStandard}
		H^p(\eta) \geq H^p(\eta') + p H^{p-1}(\eta')  \nabla H(\eta') \cdot (\eta -\eta') + \hat C_p H^p(\eta - \eta'), 
	\end{equation}
	for any $\eta, \eta' \in \R^N$. Furthermore, if $1<p<2$ we have that
	\begin{equation} \label{intro:eq:ineqStandardBis}
		H^p(\eta) \geq H^p(\eta') + p H^{p-1}(\eta')  \nabla H (\eta') \cdot (\eta-\eta') + \hat C_p [H(\eta) + H(\eta')]^{p-2}H^2(\eta - \eta'),
	\end{equation}
	for any $\eta, \eta' \in \R^N$ such that $|\eta|+|\eta'|>0$, where $\hat C_p, \tilde C_p$ are nonnegative constants.
\end{proposition}
The next proposition concerns the convergence of the gradients of a sequence in $W^{1,p}(\Omega)$. This tool is very standard in the euclidean framework, but we need its anisotropic version in the sequel.
\begin{proposition}\label{prop:convgrad}
		Let $H$ satisfy \eqref{Hh1}-\eqref{Hh3}. Let $p>1$ and let $v_k$ be a sequence of functions in $W^{1,p}(\Omega)$ such that: 
	\begin{itemize}
		\item[(a)] $v_k$ weakly converges to $v$ in $W^{1,p}(\Omega)$;
		
		\item[(b)] $\| H(\nabla v_k)\|_{L^p(\Omega)}$ converges to $\| H(\nabla v)\|_{L^p(\Omega)}$ as $k \to +\infty$.
	\end{itemize}
	
	Then $\| H(\nabla(v_k-v)) \|_{L^p(\Omega)} \to 0$ as $k \to +\infty$. 
\end{proposition}

\begin{proof}
	First we show the case $p \geq 2$. Thanks to \eqref{introeq:ineqStandard}, taking $\eta= \nabla v_k=\nabla v + \nabla(v_k-v)$ and $\eta'=\nabla v$, we get
	\begin{equation}\label{eq:convgrad1}
		\int_\Omega H^p(\nabla v_k) - \int_\Omega H^p(\nabla v) - p \int_\Omega H^{p-1}(\nabla v)  \nabla H(\nabla v) \cdot \nabla (v_k-v)  \geq \hat C_p \int_\Omega H^p(\nabla(v_k-v)) \geq 0.
	\end{equation}
	We focus our attention on the third term of \eqref{eq:convgrad1} on the left-hand side. In particular, it follows from \eqref{eq:gradHlimitato} that
	\begin{equation}\label{eq:weightp-1}
		H^{p-1}(\nabla v)  |\nabla H(\nabla v)| \leq K  H^{p-1}(\nabla v)  \in L^{p'}(\Omega).
	\end{equation}
	Therefore, we can pass to the limit as $k \to +\infty$ in \eqref{eq:convgrad1}, using (b) for the first two terms and (a) combined with \eqref{eq:weightp-1} for the third term, in order to get
	\begin{equation}\label{eq:thesisconvdiffgrad}
		\lim_{k\to +\infty} \int_\Omega H^p(\nabla(v_k-v)) = 0.
	\end{equation}
	Now, we focus our attention on the case $1<p<2$. If we assume that $\nabla v_k \rightharpoonup 0$ as $k \rightarrow +\infty$, then $\nabla v=0$ and by (b) we have $\| H(\nabla v_k) \|_{L^p(\Omega)} \rightarrow 0$ as $k \rightarrow +\infty$. Then, since $H$ is convex, we can use the Minkowski inequality in order to deduce that
	\begin{equation}\label{eq:convp<2_1}
		0 \leq \| H(\nabla(v_k-v)) \|_{L^p(\Omega)} \leq  \| H(\nabla(v_k))\| + \|H(\nabla v) \|_{L^p(\Omega)} = \| H(\nabla v_k) \|_{L^p(\Omega)} \rightarrow 0, \text{ as } k \rightarrow +\infty
	\end{equation}
	and hence the thesis. 
	
	Then we assume $\nabla v_k \rightharpoonup \nabla v \neq 0$ as $k \rightarrow +\infty$.	Thanks to \eqref{intro:eq:ineqStandardBis}, taking $\eta= \nabla v_k=\nabla v + \nabla(v_k-v)$ and $\eta'=\nabla v$, we get
\begin{equation}\label{eq:convgrad2}
	\begin{split}
	\int_\Omega H^p(\nabla v_k) - \int_\Omega H^p(\nabla v) &- p \int_\Omega H^{p-1}(\nabla v)  \nabla H(\nabla v) \cdot \nabla (v_k-v)  \\
	&\geq \hat C_p \int_\Omega \left[H(\nabla v_k)+H(\nabla v)\right]^{p-2} H^2(\nabla(v_k-v)) \geq 0.
\end{split}
\end{equation}
We can use the same argument of the previous case for all the terms on the left-hand side of \eqref{eq:convgrad2}. Thus passing to the limit as $k \rightarrow + \infty$ in \eqref{eq:convgrad2}, we get
\begin{equation}\label{eq:convp<2_2}
	\lim_{k\to +\infty} \int_\Omega \frac{H^2(\nabla(v_k-v))}{\left[H(\nabla v_k)+H(\nabla v)\right]^{2-p}} = 0.
\end{equation}
Since $1<p<2$, by triangular inequality it follows $H^{p-2}(\nabla (v_k-v)) \geq \left[H(\nabla v_k)+H(\nabla v)\right]^{p-2}$ which, gathered with \eqref{eq:convp<2_2}, implies
\eqref{eq:thesisconvdiffgrad}. Thus combining this fact with \eqref{eq:convp<2_1} we obtain the thesis.
\end{proof}

We conclude this section by denoting with $\phi_1$ the first positive eigenfunction of the anisotropic $p$-laplacian in $\Omega$. Namely $\phi_1$ satisfies (\cite{BFK,lindqv})
\begin{equation}\label{einvaluprob}
	\begin{cases} \displaystyle
		-\Delta_p^H\phi_1=  \lambda_1 \phi_1^{p-1} & \text{in}\,\,\Omega,  \\
		\phi_1=0 & \text{on}\,\,\partial \Omega.
	\end{cases}
\end{equation}
%
%
	In particular, we will use that $\phi_1 \in C^{1,\sigma}(\overline{\Omega})$ for some $0<\sigma<1$. This finally  allows to apply the Hopf boundary lemma \cite[Theorem 4.5]{CSR} to $\phi_1$ which, in particular, also permits to deduce that
	\begin{equation}\label{phi_integrabilita}
		\int_{\Omega} \phi_{1}^{r} < +\infty \quad \Longleftrightarrow \quad r > -1.
	\end{equation}
	\begin{remark}\label{rmk:regularity}
Let us now spend a few words concerning the regularity up to the boundary of $\phi_1$. As pointed out in \cite[Theorem 1.5]{CFV16}, assumption \eqref{Hh3}, i.e. to ask that $H$ is uniformly elliptic with constant $\lambda$, is equivalent to require that there exist positive constants $\Upsilon, \Gamma$ such that, for any $\xi \in \mathbb{R}^N \setminus \{0\}$, $\zeta \in \mathbb{R}^N$, the following conditions hold:
\begin{equation*}\label{eq:Hboundbelow}
	\left[D^2 H^{p}(\xi)\right]_{ij} \zeta_i \zeta_j \geq \Upsilon |\xi|^{p-2} |\zeta|^2,
\end{equation*}
and
\begin{equation*}\label{eq:Hboundabove}
	\sum_{i,j=1}^{n} \left|\left[D^2 H^{p}(\xi)\right]_{ij}\right| \leq \Gamma |\xi|^{p-1}.
\end{equation*}
The last two conditions are exactly the assumption (0.3a) and (0.3b) in \cite[Theorem 1]{lieberman}. Assumption (0.3c) is automatically fulfilled since our operator does not depend directly on $x$ and $\phi_1$. The same holds for assumption (0.3d) since in the nonlinearity does not appear the gradient of $\phi_1$. Recalling further that $\phi_1 \in L^{\infty}(\Omega)$ (see \cite{BFK}), we can finally apply \cite[Theorem 1]{lieberman} in order to deduce that $\phi_1 \in C^{1,\sigma}(\overline{\Omega})$. We also refer to the recent \cite{An} for similar regularity results on a more general class of problems.
	\end{remark}

\section{Uniqueness of weak solutions: Proof of Theorem \ref{intro:thm:uniqueness}}\label{sec:Uniqueness}
The aim of this section is to prove the uniqueness of positive weak solutions to
\begin{equation}\label{eq:Pb_section_notion}
	\begin{cases}
		\displaystyle -\Delta_p^H u= F(x,u) & \text { in } \Omega, \\
		u=0 & \text { on } \partial \Omega.
	\end{cases}
\end{equation}
where $1<p<N$ and the function $F:\Omega\times (0,+\infty)\to [0,+\infty)$ is a Carath\'eodory function that is,  for almost every $x \in \Omega$, the map $s \mapsto F(x,s)$ is continuous while for every fixed $s \in \mathbb{R}$, the function $x \mapsto F(x,s)$ is measurable. 

\medskip

First we precisely set what we mean by a weak solution to \eqref{eq:Pb_section_notion}. 

\begin{defin} \label{defweak}
	 A positive function $u\in W^{1,p}_0(\Omega)$ is a weak solution to \eqref{eq:Pb_section_notion} if $F(x,u)\in L^1_{\rm loc}(\Omega)$ and if
	\begin{equation}
		\label{defweakformulation}
		\displaystyle \int_\Omega H^{p-1}(\nabla u) \nabla H(\nabla u) \cdot \nabla \varphi=\int_\Omega F(x,u)\varphi, \ \forall \varphi \in C^1_c(\Omega).
	\end{equation}
\end{defin} 
Next lemma shows that the weak formulation \eqref{defweakformulation} can be extended to a larger class of test functions.
\begin{lemma} \label{lem:extensiontest}
	Let $H$ satisfy \eqref{Hh1}-\eqref{Hh3}, let $1<p<N$ and let $u$ be a weak solution to \eqref{eq:Pb_section_notion} in the sense of Definition \ref{defweak}. Then it holds
\[
\int_{\Omega} H^{p-1}(\nabla u)\nabla H(\nabla u)\cdot\nabla\psi = \int_{\Omega}F(x,u)\psi, \quad \forall \psi \in W_0^{1,p}(\Omega).
\]
\end{lemma}

\begin{proof}
	Let $\psi \in W_0^{1,p}(\Omega)$ be nonnegative and let us consider  $\psi_{\delta,n} = \rho_\delta \ast (\psi \wedge \phi_n)$ (where we define $\psi \wedge \phi_n:= \inf (\psi,\phi_n)$), where $\rho_\delta$ is a smooth mollifier and $0\le \phi_n\in C^1_c(\Omega)$ converging to $\psi$ in $W_0^{1,p}(\Omega)$.
	In particular it holds
	\begin{equation}\label{propapprox}
		\begin{cases}
			\psi_{\delta,n} \stackrel{\delta  \to 0^+}{\to} \psi_{n}  \ \ \ \text{in } W^{1,p}_0(\Omega) \text{ and } \text{$*$-weakly in } L^\infty(\Omega), 
			\\
			\psi_{n} \stackrel{n  \to +\infty}{\to} \psi \ \ \ \ \ \text{in } W_0^{1,p}(\Omega),
			\\
			\supp \psi_n\subset \subset \Omega: 0\le \psi_n\le \psi \ \ \ \text{for all } n\in \mathbb{N}.
		\end{cases}
	\end{equation}
	Hence we can take $\psi_{\delta,n}$ as a test function in \eqref{defweakformulation} yielding to  
	$$
	\int_{\Omega}H^{p-1}(\nabla u)\nabla H(\nabla u)\cdot\nabla\psi_{\delta,n} = \int_{\Omega}F(x,u)\psi_{\delta,n},
	$$
	and the aim is to take the limit first as $\delta \to 0^+$ and then as $n \to +\infty$. As $\psi_{\delta,n}$ converges to $\psi_n$ in $W_0^{1,p}(\Omega)$ and recalling \eqref{controlloH}, one can simply pass to the limit on the left-hand side. For the right-hand side, since $F(x,u) \in L^1_{\rm loc}(\Omega)$ and since $\psi_{\delta,n}$ converges $*$-weakly in $L^\infty(\Omega)$ to $\psi_n$ which is compactly  supported in $\Omega$, then we can pass to the limit as $\delta \rightarrow 0^+$. Hence we have shown
	\begin{equation}\label{est1}
		\int_{\Omega}H^{p-1}(\nabla u)\nabla H(\nabla u)\cdot\nabla\psi_n = \int_{\Omega}F(x,u)\psi_n.
	\end{equation}
	Now observe that, as \eqref{controlloH} and \eqref{eq:gradHlimitato} are in force, an application of the Young inequality provides 
	\begin{equation*}\label{uni3}
		\int_{\Omega}F(x,u)\psi_n \le \frac{\beta K(p-1)}{p}\int_{\Omega} |\nabla u|^{p} + \frac{\beta K}{p}\int_{\Omega}|\nabla \psi_n|^p,	
	\end{equation*}	
	and, by \eqref{propapprox}, the right-hand of the previous is bounded with respect to $n$. Then one can apply the Fatou Lemma with respect to $n$, in order to get
	\begin{equation*}\label{uni4}
		\int_{\Omega}F(x,u)\psi \le C.	
	\end{equation*}	
	Now  we can easily pass to the limit as $n\to + \infty$ in \eqref{est1}. Indeed the term on the left-hand once again passes to the limit as \eqref{controlloH} is in force and thanks to \eqref{propapprox}. For the right-hand it is sufficient an application of the Lebesgue Theorem since 
	$$F(x,u)\psi_n\le F(x,u)\psi \in L^1(\Omega).$$
	Therefore we have shown  
	\begin{equation*}
\int_{\Omega} H^{p-1}(\nabla u)\nabla H(\nabla u)\cdot\nabla\psi = \int_{\Omega}F(x,u)\psi,
	\end{equation*} 		 
	for every nonnegative $\psi\in W^{1,p}_0(\Omega)$. The proof concludes as the case of $\psi$ with general sign simply follows. 
	\end{proof}

The main goal of this section is proving Theorem \ref{intro:thm:uniqueness} which will follows as a simple application of a comparison principle. In particular let us introduce 
\begin{equation}\label{eq:Pb_per_comparison}
	\begin{cases}
		\displaystyle -\Delta_p^H v_i= G_{i}(x,v_i) & \text { in } \Omega, \\
		v_i=0 & \text { on } \partial \Omega,
	\end{cases}
\end{equation}
\noindent where $G_1, G_2 : \Omega \times (0,+\infty) \to [0,+\infty)$ are Carath\'{e}odory functions and $v_1, v_2$ are weak solutions to \eqref{eq:Pb_per_comparison}.

The proof of the following result is reminiscent of a famous result by Brezis and Oswald \cite{Brezis_Oswald} and it is an adaptation to the anisotropic setting of \cite[Theorem 2.2]{DurOl}.
\begin{proposition}\label{thm:Comparison_Principle}
		Let $H$ satisfy \eqref{Hh1}-\eqref{Hh3} and let $1<p<N$. Let  $G_1, G_2 : \Omega \times (0,+\infty) \to [0,+\infty)$ be Carath\'{e}odory functions such that either
	\begin{equation*}
		\dfrac{G_{1}(x,s)}{s^{p-1}} \textrm{ or } \dfrac{G_{2}(x,s)}{s^{p-1}} \textrm{ is strictly decreasing w.r.t. } s \textrm{ for a.e. } x \in \Omega,
	\end{equation*}
	\noindent and 
	\begin{equation}\label{eq:Ipotesi_su_G_i}
		G_{1}(x,s)\leq G_{2}(x,s) \quad \textrm{ for a.e. } x \in \Omega \textrm{ and for all } s \in (0,+\infty).
	\end{equation}
	Let $v_1$ and $v_2$ be weak solutions to \eqref{eq:Pb_per_comparison} with data, respectively, $G_{1}$ and $G_{2}$. Then $v_1 \le v_2$ almost everywhere in $\Omega$.
\end{proposition}

We split the proof in several technical lemmas. We firstly introduce some notations; let $v_1, v_2 \in W^{1,p}_{0}(\Omega)$, then for every $\varepsilon > 0$ and $k \in \N$ 
\begin{equation*}\label{A_k}
	A_{k,\epsilon}:=\{x\in\Omega:0\le(v_{1}(x)+\varepsilon)^{p}-(v_{2}(x)+\varepsilon)^{p}\le k\}, \ A_k:=A_{k,0}, \ A_{k,\varepsilon}^{c}:=\Omega \setminus A_{k,\epsilon}, \ A^c_k:=A^c_{k,0}.
\end{equation*}
\begin{lemma}\label{lem:giuste_test}
	Let $H$ satisfy \eqref{Hh1}-\eqref{Hh3} and let $1<p<N$. Let $v_1, v_2 \in W^{1,p}_{0}(\Omega)$ and let us denote by
	\begin{equation}\label{eq:Def_psi_1}
		\psi_{1}:= \dfrac{T_k (((v_{1}+\varepsilon)^{p}-(v_{2}+\varepsilon)^{p})^{+})}{(v_{1}+\varepsilon)^{p-1}},
	\end{equation}
	\noindent and 
	\begin{equation}\label{eq:Def_psi_2}
		\psi_{2}:= \dfrac{T_k (((v_{1}+\varepsilon)^{p}-(v_{2}+\varepsilon)^{p})^{+})}{(v_{2}+\varepsilon)^{p-1}}.
	\end{equation}
	Then, $\psi_{1}, \psi_{2}\in W^{1,p}_{0}(\Omega)$ and
	\begin{equation*}
		\begin{aligned}
			\nabla \psi_{1} =& \left(\nabla v_1 - p \left(\dfrac{v_2 + \varepsilon}{v_1+\varepsilon}\right)^{p-1} \nabla v_2 + (p-1) \left(\dfrac{v_2+\varepsilon}{v_1+\varepsilon}\right)^p \nabla v_1 \right)\chi_{A_{k,\varepsilon}}\\
			&-(p-1)\dfrac{T_k (((v_{1}+\varepsilon)^{p}-(v_{2}+\varepsilon)^{p})^{+})}{(v_{1}+\varepsilon)^{p}}\nabla v_1 \, \chi_{A_{k,\varepsilon}^{c}}
		\end{aligned}
	\end{equation*}
	\noindent and 	
	\begin{equation*}
		\begin{aligned}
			\nabla \psi_{2} =& -\left(\nabla v_2 - p \left(\dfrac{v_1 + \varepsilon}{v_2+\varepsilon}\right)^{p-1} \nabla v_1 + (p-1) \left(\dfrac{v_1+\varepsilon}{v_2+\varepsilon}\right)^p \nabla v_2 \right)\chi_{A_{k,\varepsilon}}\\
			&-(p-1)\dfrac{T_k (((v_{1}+\varepsilon)^{p}-(v_{2}+\varepsilon)^{p})^{+})}{(v_{2}+\varepsilon)^{p}}\nabla v_2 \, \chi_{A_{k,\varepsilon}^{c}}.
		\end{aligned}
	\end{equation*}
	\begin{proof}
		We refer to \cite[Remark 2.5]{DurOl} for the details showing that $\psi_{1}, \psi_{2}\in W^{1,p}_{0}(\Omega)$. The rest are straightforward computations.
	\end{proof}
\end{lemma}

\begin{lemma}\label{lem:Uso_di_convessità}
	Let $H$ satisfy \eqref{Hh1}-\eqref{Hh3} and let $1<p<N$. Let $v_1, v_2 \in W^{1,p}_{0}(\Omega)$. Then
	\begin{equation}\label{eq:Stima_1_in_A_k_eps}
		H^{p}(\nabla v_1) - \left(\dfrac{v_1+\varepsilon}{v_2+\varepsilon}\right)^p H^{p}(\nabla v_2) - p \left(\dfrac{v_1+\varepsilon}{v_2+\varepsilon}\right)^{p-1} H^{p-1}(\nabla v_2) \nabla H(\nabla v_2)\cdot \left(\nabla v_1 - \left(\dfrac{v_1+\varepsilon}{v_2+\varepsilon}\right)\nabla v_2\right) \geq 0,
	\end{equation}
	\noindent and
	\begin{equation}\label{eq:Stima_2_in_A_k_eps}
		H^{p}(\nabla v_2) - \left(\dfrac{v_2+\varepsilon}{v_1+\varepsilon}\right)^p H^{p}(\nabla v_1) - p \left(\dfrac{v_2+\varepsilon}{v_1+\varepsilon}\right)^{p-1} H^{p-1}(\nabla v_1) \nabla H(\nabla v_1)\cdot \left(\nabla v_2 - \left(\dfrac{v_2+\varepsilon}{v_1+\varepsilon}\right)\nabla v_1\right)\geq 0.
	\end{equation}
	\begin{proof}
		The validity of \eqref{eq:Stima_1_in_A_k_eps} is an immediate consequence of \eqref{introeq:ineqStandard} and \eqref{intro:eq:ineqStandardBis} with $\eta = \nabla v_1$ and $\eta' = \left(\tfrac{v_1 +\varepsilon}{v_2 +\varepsilon}\right)\nabla v_2$ (and $\hat{C}_p =0$), while \eqref{eq:Stima_2_in_A_k_eps} follows as well from \eqref{introeq:ineqStandard} and \eqref{intro:eq:ineqStandardBis} choosing $\eta = \nabla v_2$ and $\eta' = \left(\tfrac{v_2 +\varepsilon}{v_1 +\varepsilon}\right)\nabla v_1$ (and $\hat{C}_p =0$). This closes the proof.
	\end{proof}
\end{lemma}	

\begin{lemma}\label{lem:limite_in_epsilon}
	Let $H$ satisfy \eqref{Hh1}-\eqref{Hh3} and let $1<p<N$. Let  $G : \Omega \times (0,+\infty) \to [0,+\infty)$ be a Carath\'{e}odory function such that 
\begin{equation*}
	\dfrac{G(x,s)}{s^{p-1}}  \textrm{ is strictly decreasing w.r.t. } s \textrm{ for a.e. } x \in \Omega.
\end{equation*}
	For every $\varepsilon\geq 0$, define
	\begin{equation}\label{eq:Def_r_k_varepsilon}
		\begin{aligned}
			r_{k,\varepsilon}:= & (p-1)\dfrac{T_k (((v_{1}+\varepsilon)^{p}-(v_{2}+\varepsilon)^{p})^{+})}{(v_{1}+\varepsilon)^{p}}H^{p}(\nabla v_1) \, \chi_{A_{k,\varepsilon}^{c}} \\
			&+ \left(\dfrac{G(x,v_1)}{(v_1+\varepsilon)^{p-1}} - \dfrac{G(x,v_2)}{(v_2+\varepsilon)^{p-1}}\right)T_k (((v_{1}+\varepsilon)^{p}-(v_{2}+\varepsilon)^{p})^{+}),
		\end{aligned}
	\end{equation}
	\noindent and 
	\begin{equation}\label{eq:Def_tilde_r_k_varepsilon}
		\begin{aligned}
			\tilde{r}_{k,\varepsilon}:= & (p-1)\dfrac{T_k (((v_{1}+\varepsilon)^{p}-(v_{2}+\varepsilon)^{p})^{+})}{(v_{1}+\varepsilon)^{p}}H^{p}(\nabla v_1) \, \chi_{A_{k,\varepsilon}^{c}} \\
			&+ \dfrac{G(x,v_1)}{(v_1+\varepsilon)^{p-1}} \, T_k (((v_{1}+\varepsilon)^{p}-(v_{2}+\varepsilon)^{p})^{+}).
		\end{aligned}
	\end{equation}
	Then the following holds:
	\begin{itemize}
		\item[i)] $0\leq r_{k,\varepsilon}^{+} \leq \tilde{r}_{k,\varepsilon}$;
		\item[ii)] $r_{k,\varepsilon}^{\pm} \to r_{k,0}^{\pm}$ a.e. in $\Omega$, as $\varepsilon \to 0^{+}$; 
		\item[iii)] $\tilde{r}_{k,\varepsilon} \to \tilde{r}_{k,0}$ a.e. in $\Omega$, as $\varepsilon \to 0^{+}$;
		\item[iv)] $r_{k,\varepsilon}^{+} \to r_{k,0}^{+}$ in $L^{1}(\Omega)$, as $\varepsilon \to 0^{+}$.
	\end{itemize}
	\begin{proof}
		Since
		\begin{equation*}
			(f+g)^{+} \leq f^{+} + g^{+} \quad \textrm{ and } \quad (f-g)^{+} \leq f \quad \textrm{ for every } f,g \geq 0,
		\end{equation*}
		\noindent then i) immediately follows. Regarding ii) and iii), one has simply to observe that both $v_1$ and $v_2$ are positive functions. \\
		\noindent We are left with iv). Observe that, as $T_{k}(s)\leq s$ for every $s \geq 0$, for every $\varepsilon\geq 0$ we have either
		\begin{equation}\label{eq:Stima_1_per_tilde_r}
			\dfrac{T_k (((v_{1}+\varepsilon)^{p}-(v_{2}+\varepsilon)^{p})^{+})}{(v_{1}+\varepsilon)^{p}} \,\chi_{A_{k,\varepsilon}^{c}} \leq 1 
		\end{equation}
		\noindent and 
		\begin{equation}\label{eq:Stima_2_per_tilde_r}
			\dfrac{T_k (((v_{1}+\varepsilon)^{p}-(v_{2}+\varepsilon)^{p})^{+})}{(v_{1}+\varepsilon)^{p-1}}  \leq \dfrac{(v_1+\varepsilon)^{p}-\varepsilon^{p}}{(v_{1}+\varepsilon)^{p-1}} < p \, v_1,
		\end{equation}
		\noindent where we used Mean Value Theorem. Moreover, as $v_1 \in W^{1,p}_{0}(\Omega)$ is a weak solution of \eqref{eq:Pb_per_comparison} with data $G$ and as Lemma \ref{lem:giuste_test} is in force, we can use $v_1$ as a test function in the weak formulation of \eqref{eq:Pb_per_comparison}, yielding to
		\begin{equation}\label{eq:Test_con_v_1}
			\int_{\Omega}H^{p}(\nabla v_1) = \int_{\Omega}H^{p-1}(\nabla v_1) \nabla H (\nabla v_1)\cdot \nabla v_1 = \int_{\Omega} G(x,v_1)v_1.
		\end{equation}
		Combining \eqref{eq:Stima_1_per_tilde_r}, \eqref{eq:Stima_2_per_tilde_r} and \eqref{eq:Test_con_v_1}, we get
		\begin{equation*}
			\tilde{r}_{k,\varepsilon} \leq (p-1) H^{p}(\nabla v_1) + p \, G(x,v_1)\, v_1 \in L^{1}(\Omega).
		\end{equation*}
		The latter is enough to employ the Lebesgue Dominated Convergence Theorem yielding that
		\begin{equation*}
			\tilde{r}_{k,\varepsilon} \to \tilde{r}_{k,0} \textrm{ in } L^{1}(\Omega) \textrm{ as } \varepsilon \to 0^{+}.
		\end{equation*}	
		Now, recalling i), we can apply Vitali Theorem obtaining iv). This closes the proof.
	\end{proof}
\end{lemma}

We are now ready to prove the Comparison Principle stated in Proposition \ref{thm:Comparison_Principle}.
\begin{proof}[Proof of Proposition \ref{thm:Comparison_Principle}]
	Without loss of generality, we start by assuming
	\begin{equation}\label{eq:Ipotesi_su_G_1}
		\dfrac{G_{1}(x,s)}{s^{p-1}} \textrm{ is strictly decreasing w.r.t. } s \textrm{ for a.e. } x \in \Omega.
	\end{equation}
 Indeed the case in which is  $G_{2}(x,s)s^{1-p}$ to be strictly decreasing is a simple adaption of the proof we present here. 
	
	We consider the functions $\psi_1, \psi_2 \in W^{1,p}_{0}(\Omega)$ defined, respectively, in \eqref{eq:Def_psi_1} and \eqref{eq:Def_psi_2} and which, thanks to Lemma \ref{lem:extensiontest}, can be taken as test
	in the equations solved, respectively, by $v_1$ and $v_2$. This yields to
	\begin{equation}\label{eq:Solved_by_v_1}
		\int_{\Omega}H^{p-1}(\nabla v_1)\nabla H(\nabla v_1)\cdot \nabla \psi_1 = \int_{\Omega}G_{1}(x,v_1)\psi_1
	\end{equation}
	\noindent and 
	\begin{equation}\label{eq:Solved_by_v_2}
		\int_{\Omega}H^{p-1}(\nabla v_2)\nabla H(\nabla v_2)\cdot \nabla \psi_2 = \int_{\Omega}G_{2}(x,v_2)\psi_2.
	\end{equation}
	Exploiting the expressions of $\nabla \psi_1$ and $\nabla \psi_2$ from Lemma \ref{lem:giuste_test}, and subtracting \eqref{eq:Solved_by_v_2} from \eqref{eq:Solved_by_v_1}, after a rearrangement of terms, we find 
	\begin{equation*}
		\begin{aligned}
			&\int_{A_{k,\varepsilon}}\left( H^{p}(\nabla v_1) - \left(\dfrac{v_1+\varepsilon}{v_2+\varepsilon}\right)^p H^{p}(\nabla v_2) - p \left(\dfrac{v_1+\varepsilon}{v_2+\varepsilon}\right)^{p-1} H^{p-1}(\nabla v_2) \nabla H(\nabla v_2)\cdot \left(\nabla v_1 - \left(\dfrac{v_1+\varepsilon}{v_2+\varepsilon}\right)\nabla v_2\right)\right)\\
			&+\int_{A_{k,\varepsilon}}\left( H^{p}(\nabla v_2) - \left(\dfrac{v_2+\varepsilon}{v_1+\varepsilon}\right)^p H^{p}(\nabla v_1) - p \left(\dfrac{v_2+\varepsilon}{v_1+\varepsilon}\right)^{p-1} H^{p-1}(\nabla v_1) \nabla H(\nabla v_1)\cdot \left(\nabla v_2 - \left(\dfrac{v_2+\varepsilon}{v_1+\varepsilon}\right)\nabla v_1\right)\right)\\
			&+ (p-1) \int_{A_{k,\varepsilon}^{c}}\left(\dfrac{T_k (((v_{1}+\varepsilon)^{p}-(v_{2}+\varepsilon)^{p})^{+})}{(v_{2}+\varepsilon)^{p}} H^{p}(\nabla v_2) - \dfrac{T_k (((v_{1}+\varepsilon)^{p}-(v_{2}+\varepsilon)^{p})^{+})}{(v_{1}+\varepsilon)^{p}}H^{p}(\nabla v_1) \right)\\
			&\leq \int_{\Omega} \left(\dfrac{G_{1}(x,v_1)}{(v_1+\varepsilon)^{p-1}} - \dfrac{G_{2}(x,v_2)}{(v_2+\varepsilon)^{p-1}}\right)T_k (((v_{1}+\varepsilon)^{p}-(v_{2}+\varepsilon)^{p})^{+}).
		\end{aligned}
	\end{equation*}
	Thanks to Lemma \ref{lem:Uso_di_convessità}, it follows that the first two integrals are actually nonnegative, and this, combined with \eqref{eq:Ipotesi_su_G_i}, leads to
	\begin{equation}\label{eq:Stima_basso_int_r_k_eps}
		\begin{aligned}
			0 &\leq (p-1) \int_{A_{k,\varepsilon}^{c}}\dfrac{T_k (((v_{1}+\varepsilon)^{p}-(v_{2}+\varepsilon)^{p})^{+})}{(v_{2}+\varepsilon)^{p}} H^{p}(\nabla v_2) \\
			&\leq  (p-1) \int_{A_{k,\varepsilon}^{c}}\dfrac{T_k (((v_{1}+\varepsilon)^{p}-(v_{2}+\varepsilon)^{p})^{+})}{(v_{1}+\varepsilon)^{p}}H^{p}(\nabla v_1) \\
			& \quad + \int_{\Omega} \left(\dfrac{G_{1}(x,v_1)}{(v_1+\varepsilon)^{p-1}} - \dfrac{G_{2}(x,v_2)}{(v_2+\varepsilon)^{p-1}}\right)T_k (((v_{1}+\varepsilon)^{p}-(v_{2}+\varepsilon)^{p})^{+})\\
			&\stackrel{\eqref{eq:Ipotesi_su_G_1}}{\leq}  (p-1) \int_{A_{k,\varepsilon}^{c}}\dfrac{T_k (((v_{1}+\varepsilon)^{p}-(v_{2}+\varepsilon)^{p})^{+})}{(v_{1}+\varepsilon)^{p}}H^{p}(\nabla v_1) \\
			& \quad + \int_{\Omega} \left(\dfrac{G_{1}(x,v_1)}{(v_1+\varepsilon)^{p-1}} - \dfrac{G_{1}(x,v_2)}{(v_2+\varepsilon)^{p-1}}\right)T_k (((v_{1}+\varepsilon)^{p}-(v_{2}+\varepsilon)^{p})^{+})\\
			&= \int_{\Omega}r_{k,\varepsilon} = \int_{\Omega} (r_{k,\varepsilon}^{+}-r_{k,\varepsilon}^{-}),
		\end{aligned}
	\end{equation}
	\noindent where $r_{k,\varepsilon}$ is as in \eqref{eq:Def_r_k_varepsilon}. Now, noticing first that by Fatou Lemma
	\begin{equation*}
		\limsup_{\varepsilon \to 0^{+}}\int_{\Omega}-r_{k,\varepsilon}^{-} \leq \int_{\Omega}-r_{k,0}^{-},
	\end{equation*}
	\noindent and recalling both \eqref{eq:Stima_basso_int_r_k_eps} and Lemma \ref{lem:limite_in_epsilon}-iv), we find
	\begin{equation*}
		0\leq \limsup_{\varepsilon \to 0^{+}}\int_{\Omega}(r_{k,\varepsilon}^{+}-r_{k,\varepsilon}^{-}) \leq \int_{\Omega}(r_{k,0}^{+}-r_{k,0}^{-}) = \int_{\Omega}r_{k,0},
	\end{equation*}
	\noindent which can be explicitly written as
	\begin{equation}\label{eq:Post_limite_in_epsilon}
		0\leq \int_{\Omega}(p-1)\dfrac{T_{k}((v_1^p-v_2^p)^{+})}{v_1^p}H^{p}(\nabla v_1)\chi_{A_{k,\varepsilon}^{c}}
		+ \int_{\Omega}\left(\dfrac{G_{1}(x,v_1)}{v_1^{p-1}}-\dfrac{G_{1}(x,v_2)}{v_2^{p-1}}\right)T_{k}((v_1^p-v_2^p)^{+}).
	\end{equation}
	In order to pass to the limit as $k \to +\infty$ in \eqref{eq:Post_limite_in_epsilon}, we first notice that 
	\begin{equation*}
		\chi_{A_{k,\varepsilon}^{c}} \to 0 \quad \textrm{ a.e. in } \Omega \textrm{ as } k \to +\infty.
	\end{equation*}
	Combining the latter with \eqref{eq:Stima_1_per_tilde_r} (with $\varepsilon =0$), we can use Lebesgue Dominated Convergence Theorem to get
	\begin{equation}\label{eq:Limite_in_k_1}
		\dfrac{T_{k}((v_1^p-v_2^p)^{+})}{v_1^p}H^{p}(\nabla v_1)\chi_{A_{k,\varepsilon}^{c}} \to 0 \quad \textrm{ in } L^{1}(\Omega) \textrm{ as } k \to +\infty.
	\end{equation}	
	We further notice that \eqref{eq:Ipotesi_su_G_1} readily implies that
	\begin{equation}\label{eq:Consequence_of_monotonicity}
		0\leq -\left(\dfrac{G_{1}(x,v_1)}{v_{1}^{p-1}}-\dfrac{G_{1}(x,v_2)}{v_{2}^{p-1}}\right) T_{k}((v_1^p-v_2^p)^{+}).
	\end{equation}
	The increasing monotonicity in $k$ of the right-hand of \eqref{eq:Consequence_of_monotonicity} allows to use Monotone Convergence Theorem yielding to
	\begin{equation}\label{eq:Limite_in_k_2}
		\lim_{k\to +\infty} \int_{\Omega} \left(\dfrac{G_{1}(x,v_1)}{v_1^{p-1}}-\dfrac{G_{1}(x,v_2)}{v_2^{p-1}}\right)T_{k}((v_1^p-v_2^p)^{+}) =  \int_{\Omega} \left(\dfrac{G_{1}(x,v_1)}{v_1^{p-1}}-\dfrac{G_{1}(x,v_2)}{v_2^{p-1}}\right)(v_1^p-v_2^p)^{+}.
	\end{equation}
	Thus, using both \eqref{eq:Limite_in_k_1} and \eqref{eq:Limite_in_k_2}, we can pass to the limit in $k$ in \eqref{eq:Post_limite_in_epsilon} finding
	\begin{equation*}\label{eq:Stima_da_un_lato}
		0\leq \int_{\Omega} \left(\dfrac{G_{1}(x,v_1)}{v_1^{p-1}}-\dfrac{G_{1}(x,v_2)}{v_2^{p-1}}\right)(v_1^p-v_2^p)^{+}.
	\end{equation*}
	On the other hand, \eqref{eq:Ipotesi_su_G_1} implies that the reverse inequality holds as well, namely
	\begin{equation*}
		\int_{\Omega} \left(\dfrac{G_{1}(x,v_1)}{v_1^{p-1}}-\dfrac{G_{1}(x,v_2)}{v_2^{p-1}}\right)(v_1^p-v_2^p)^{+}\leq 0,
	\end{equation*}
	\noindent and this is enough to conclude that
	\begin{equation*}
		(v_1^p-v_2^p)^{+} \equiv 0 \quad \textrm{ a.e. in } \Omega,
	\end{equation*}
	\noindent and this closes the proof.
\end{proof}

The proof of Theorem \ref{intro:thm:uniqueness} is now an immediate consequence of Proposition \ref{thm:Comparison_Principle}.

\section{Existence of solutions}\label{sec:Existence}

In this section we deal with existence of weak solutions to the model problem \eqref{PbMain}.
We notice that the choice
\begin{equation}\label{Fmodello}
	F(x,s)= \frac{f(x)}{s^\gamma} + h(x) s^\theta.
\end{equation}
(where $\gamma>0$, $0<\theta<p-1$ and $f,h$ are measurable functions belonging to some suitable Lebesgue spaces), satisfies \eqref{intro:eq:Ipotesi_su_F} so that an application of Theorem \ref{intro:thm:uniqueness} will assure that the weak solution to \eqref{PbMain} we find is actually the unique one.

%

For the sake of clarity, we now restate the definition of solution.
\begin{defin}
	\label{defweakmod}
	A positive function $u\in W^{1,p}_0(\Omega)$ is a weak solution to \eqref{PbMain} if $fu^{-\gamma}, h u^\theta \in L^1_{\rm loc}(\Omega)$ and if
	\begin{equation}
		\label{defweakformulationmod}
		\displaystyle \int_\Omega H^{p-1}(\nabla u) \nabla H(\nabla u) \cdot \nabla \varphi=\int_\Omega \left(\frac{f}{u^\gamma} + h u^\theta \right)\varphi, \ \forall \varphi \in C^1_c(\Omega).
	\end{equation}
\end{defin}


We recall that $\phi_1$ will denote the first eigenfunction associated to the anisotropic $p$-laplacian.

\medskip

As already mentioned in the introduction, we first prove the existence of solutions when the functions $f$ and $h$ are both possibly unbounded and $\gamma \le 1$.
Recall also that we set $(\frac{p^*}{1-\gamma})':=1$ if $\gamma=1$.
\begin{theorem}	\label{intro:thm:existencegammaminoredi1}
	Let $H$ satisfy \eqref{Hh1}-\eqref{Hh3} and let $1<p<N$. Assume that $0\le f \in L^{(\frac{p^*}{1-\gamma})'}(\Omega)$, $0\le h \in L^{(\frac{p^*}{1+\theta})'}(\Omega)$, $0 < \gamma \le 1$ and $0 \le \theta < p-1$. Then problem \eqref{PbMain} admits a unique weak solution $u \in W^{1,p}_0(\Omega)$.
\end{theorem}   
A second interest deals with the behaviour of bounded solutions to \eqref{PbMain} when both $f,h$ are bounded.
This result is a suitable extension of the one contained in \cite{lazer}.
In order to derive the estimates we need in the non-linear case, borrowing their approach and ideas, we re-adapt their argument also inspired by \cite{ES, OPsurvey}. We also refer to \cite{MonSciuTro} for the case of anisotropic $p$-laplacian where the authors prove an analogous result when $f=1$ and in presence of the Hardy potential.
\begin{theorem}\label{thm:estbelowabove} 
	Let $H$ satisfy \eqref{Hh1}-\eqref{Hh3} and let $1<p<N$. Let $f \in L^\infty(\Omega)$ such that $f>c>0$ for some positive constant $c$, let $0 \leq h \in L^\infty(\Omega)$, $0\le \gamma<2 +1/(p-1)$, $\theta\ge 0$.
	Let $u$ be a bounded weak solution $u$ to 
	\eqref{PbMain}. Then there exist two positive constants $s_1$, $s_2$ such that
	\begin{equation} \label{estimateOnU}
		\begin{cases}
			s_1 \phi_1(x)^{\frac{p}{\gamma+p-1}} \leq u(x) \leq s_2
			\phi_1(x)^{\frac{p}{\gamma+p-1}} \ &\text{if }\gamma>1,
			\\
			s_1 \phi_1(x) \leq u(x) \leq s_2
			\phi_1(x)^t \ &\text{if }\gamma\le1, 
		\end{cases}
	\end{equation}
	for almost every $x \in
	\Omega$ and for any $\displaystyle t\in \left(\frac{p-1}{p},1\right)$.
\end{theorem}

\begin{remark}
Theorem \ref{thm:estbelowabove} applies to bounded solutions and we will employ it to approximate problems where this property is guaranteed. By the way, it is worth to remark that, even if the  $L^\infty$-regularity is outside the scope of the present paper, the solution to \eqref{PbMain} can be shown to be bounded under standard hypotheses on the data.
Indeed, if $f,h \in L^s(\Omega)$ with $s > N/p$, the unique weak solution is bounded and this can be shown by suitably adapting the truncation technique detailed in \cite{BoOr} which is shown for the purely $p$-sublinear case. In particular, the argument relies on using $G_k(u)$ as test function in the weak formulation to derive decay estimates for the super-level sets, eventually invoking Stampacchia’s lemma. 
\end{remark}

\subsection{A Variational approach for case $\gamma>1$}

One more result concerns a variational approach to the existence of weak solutions to \eqref{PbMain}. The result we present in this section is an extension of the one given in \cite{sz} in case of the laplacian operator. 
In particular, a necessary and sufficient condition is shown to hold to have a unique weak solution to \eqref{PbMain} when $0<f\in L^1(\Omega)$ is such that
\begin{equation}\label{eq:compatibility}
	\int_{\Omega} f|u_0|^{1-\gamma}  < +\infty,
\end{equation}
for some $u_0 \in W^{1,p}_0(\Omega)$. Let us also stress that the next theorem will be used to prove the existence result given  by Theorem \ref{intro:thm:existence1} and \ref{intro:thm:existence2}.

\begin{theorem}\label{thmexistence1}
	Let $H$ satisfy \eqref{Hh1}-\eqref{Hh3} and let $1<p<N$. Let $\gamma>1$, $0<f\in L^1(\Omega)$, $0\leq h \in L^\infty(\Omega)$, $0\leq \theta<p-1$.
	Then there exists a unique weak solution to \eqref{PbMain} if and only if there exists a function $u_0\in W^{1,p}_0(\Omega)$ such that  \eqref{eq:compatibility} yields.
\end{theorem}

\begin{proof}
	First, we show that if $u$ is a weak solution to \eqref{PbMain}, then  \eqref{eq:compatibility} holds with $u_0 = u$. 
	
	To this end, we observe that, by Lemma \ref{lem:extensiontest}, one can choose $u$ as a test function in the weak formulation of \eqref{PbMain}, obtaining 
	\begin{equation*}
		\int_\Omega H^{p-1}(\nabla u)\nabla H(\nabla u) \cdot \nabla u = \int_\Omega f u^{1-\gamma} \, + \int_\Omega h u^{\theta+1}.
	\end{equation*}
	We point out that the second term on the right-hand of the previous is finite, since $0 \le \theta<p-1$ and $0 \leq  h\in L^\infty(\Omega)$, and by H\"older inequality with conjugate exponents $(p/(\theta+1),p/(p-1-\theta))$ it holds
	\begin{equation}\label{eq:bddk}
		\int_\Omega h u^{\theta+1} \leq \|h\|_{L^\infty(\Omega)} |\Omega|^{\frac{p-1-\theta}{p}} \|u\|_{L^p(\Omega)}^{\theta+1} < +\infty.
	\end{equation}
	Using \eqref{controlloH} and \eqref{eq:eulero}, we get
	\begin{equation*}
		\int_\Omega f u^{1-\gamma} \, = \int_\Omega H^{p}(\nabla u) - \int_\Omega h u^{\theta+1} \leq \int_\Omega H^{p}(\nabla u) \leq \beta^p \int_\Omega |\nabla u|^p< +\infty		
	\end{equation*}
	which is \eqref{eq:compatibility} with $u_0=u$. 
	
	\medskip

	Conversely we assume that there exists $u_0\in W^{1,p}_0(\Omega)$ satisfying \eqref{eq:compatibility} and we focus on proving the existence of a unique solution $u\in W^{1,p}_0(\Omega)$ to \eqref{PbMain}. 
	Let us stress that, in order to prove the sufficient part of the proof, the strategy relies on a variational approach: we look for solutions as critical points of the associated energy functional $J$ constrained to the manifold $\mathcal{N}$ which are defined by
	$$J(u) := \frac{1}{p}\int_{\Omega}H^p(\nabla u) + \frac{1}{\gamma-1}\int_{\Omega} f|u|^{1-\gamma} - \frac{1}{\theta+1} \int_\Omega h|u|^{\theta+1},$$
	and 
	\begin{equation*}
		\mathcal{N}:=\left\{u\in W^{1,p}_0(\Omega) : \int_{\Omega}H^p(\nabla u)  - \int_{\Omega} f|u|^{1-\gamma} - \int_\Omega h|u|^{\theta+1} \, \ge 0 \right\}.
	\end{equation*}	
	Moreover we set
	\begin{equation*}
		\mathcal{N^*}:=\left\{u\in W^{1,p}_0(\Omega) :  \int_{\Omega} H^p(\nabla u)  - \int_{\Omega} f|u|^{1-\gamma} - \int_\Omega h|u|^{\theta+1} \, = 0 \right\},
	\end{equation*}
	which is a natural constraint to which any weak solution belongs.
	
	\medskip
	
	Let us point out that, from here on, the idea is to apply the Ekeland Variational Principle to secure a minimizing sequence $u_n$ for $J$ belonging to $\mathcal{N}$. Then the aim is to show that this sequence is bounded in $W_0^{1,p}(\Omega)$ and, up to a subsequence, it converges weakly to a limit function $u$ which is a weak solution to \eqref{PbMain}. 
	
	\medskip
	
	We start by establishing the structural properties of the constraint  $\mathcal{N}$, namely we show that it is non-empty, unbounded, closed, and bounded away from the origin.

	First observe that, for any $u\in W^{1,p}_0(\Omega)$ such that $\int_{\Omega}f |u|^{1-\gamma}<+\infty$  and as \eqref{eq:bddk} is in force, one can analyze the energy functional along the curve $t \mapsto tu$ for $t>0$; then one has that $J(tu) \to +\infty$ as $t \to 0^+$ and as $t \to +\infty$. Consequently, there exists $t(u) > 0$ such that
	\begin{equation}\label{stimabasso}
	J(tu) \geq J(t(u)u) \quad \forall t > 0,
	\end{equation}
	and also satisfying $t(u)u \in \mathcal{N}^*$. Furthermore, for all $t \ge t(u)$, one simply has $tu \in \mathcal{N}$, showing that $\mathcal{N}$ is unbounded. Therefore, as \eqref{eq:compatibility}, one can apply the above argument to $u_0$, providing also that $\mathcal{N}^*$ is non-empty.

   The set $\mathcal{N}$ is also closed as it can be simply proved by an application of the Fatou Lemma.
	Moreover it holds
	$$\int_{\Omega} |\nabla v_n|^p \ge c>0,$$
	for any sequence $v_n \in \mathcal{N}$ where $c$ is independent of $n$.
	Indeed, assuming by contradiction that $v_n$ tends to zero in $W^{1,p}_0(\Omega)$ then, by the reverse H\"older inequality (with indexes ($1/\gamma, 1/(\gamma-1)$)), one yields to 
	\[
	0 \le \left(\int_{\Omega} f^\frac{1}{\gamma}\right)^\gamma\left(\int_{\Omega}|v_n| \right)^{1-\gamma}  \le \int_{\Omega} f|v_n|^{1-\gamma} + \int_\Omega h |v_n|^{\theta+1} \le \int_{\Omega}H^p(\nabla v_n) \stackrel{\eqref{controlloH}}{\le} \beta^p \int_{\Omega}|\nabla v_n|^p \rightarrow 0,
	\]
	as, we recall, $v_n\in \mathcal{N}$. Since $f$ is not identically zero and $\gamma>1$, the previous implies  $\int_{\Omega}|v_n|$ blows up providing a contradiction.

	Now observe that the functional $J$ is bounded from below on $\mathcal{N}$ and, thanks to the Ekeland's variational Principle, there exists a minimizing sequence $u_n\in \mathcal{N}$ such that
	\begin{equation}\label{eke1}
		\displaystyle J(u_n) \le \inf_\mathcal{N}J +\frac{1}{n},
	\end{equation}
	\begin{equation}\label{eke2}
\displaystyle J(u_n)- J(w) \le \frac{\left(\displaystyle \int_{\Omega}|\nabla (u_n- w)|^p\right)^{\frac{1}{p}}}{n}, \ \forall w\in \mathcal{N}.
	\end{equation}
	Recalling that $H$ is even, we may assume $u_n\ge 0$ as $J(|u_n|)= J(u_n)$. Furthermore, as $\int_{\Omega}fu_n^{1-\gamma} + \int_\Omega h u_n^{\theta+1} \le \beta \int_{\Omega}|\nabla u_n|^p$ (recall \eqref{controlloH}) then $u_n>0$ a.e. in $\Omega$. We also note that $u_n$ is bounded in $W^{1,p}_0(\Omega)$ with respect to $n$ as it simply follows from \eqref{eke1} and from \eqref{controlloH}. Now we finally stress that we denote by $u$ the a.e.~limit (up to a subsequence) of the sequence $u_n$ as $n\to+\infty$.
	 Moreover, also notice that it follows from \eqref{eke1} and from the Fatou Lemma that
	\begin{equation*}
		\int_{\Omega} fu^{1-\gamma} \le \liminf_{n \to \infty} \int_{\Omega} fu_n^{1-\gamma} < C,
	\end{equation*}
	where $C$ does not depend on $n$ as $u_n$ is bounded in $W^{1,p}_0(\Omega)$ with respect to $n$. In particular this implies that $u > 0$ a.e. in $\Omega$ as $\gamma > 1$ and $f>0$ a.e. in $\Omega$ and that \eqref{stimabasso} is in force.
	Assume for the moment that $u\in \mathcal{N}$; this actually implies that $t \mapsto J(tu)$ admits a unique minimum at $t=1$.	
	Therefore, by also using weak lower semicontinuity and the Fatou Lemma, we get
	\begin{equation}\label{uNstar}
		\begin{split}
			\displaystyle \inf_{\mathcal{N}}J &= \lim_{n\to +\infty}J(u_n) \ge \frac{1}{p} \int_{\Omega}H^p(\nabla u) + \frac{1}{\gamma-1}\int_{\Omega} fu^{1-\gamma} - \frac{1}{\theta+1} \int_\Omega h u^{\theta+1}\\
			&= J(u)\ge J(t(u)u) \ge \inf_{\mathcal{N}^*} J \ge \inf_{\mathcal{N}} J,
		\end{split}
	\end{equation}
	where, as $u_n$ is bounded with respect to $n$ in $W^{1,p}_0(\Omega)$, we also employed that it converges, by standard arguments and up to a subsequence, to $u$ in $L^{q}(\Omega)$ for every $q\in \left[1,\tfrac{pN}{N-p}\right)$, which allows to pass to the limit the term involving $h$. 
	Consequently, the previous inequalities become equalities and it is not difficult to show that $\| H(\nabla u_n)\|_{L^p(\Omega)}$ converges to $\| H(\nabla u)\|_{L^p(\Omega)}$ as $n \to +\infty$. Let us also stress that this allows to apply Proposition \ref{prop:convgrad} in the sequel. In particular, \eqref{uNstar} also gives that $t(u)=1$ which implies that $u \in \mathcal{N}^*$.
		
	\medskip
	
	Then, in order to conclude the previous discussion, we still need to show that $u\in \mathcal{N}$, which follows once we fix $\varphi=u$ into: 
		\begin{equation}\label{disvar}
		\int_{\Omega} H^{p-1}(\nabla u) \nabla H(\nabla u) \cdot \nabla \varphi \ge \int_{\Omega}\frac{f\varphi}{u^\gamma}+\int_{\Omega} h u^{\theta} \varphi, \ \forall \varphi \in W^{1,p}_{0}(\Omega) \text{ with } \varphi\ge 0.
	\end{equation}
	We prove \eqref{disvar} by distinguishing two cases, depending on whether (up to a subsequence) the minimizing sequence definitely lies entirely in the interior $\mathcal{N} \setminus \mathcal{N}^*$ or on the boundary $\mathcal{N}^*$.
	
	\medskip

	\emph{i) $u_n$ belongs to $\mathcal{N}\setminus\mathcal{N}^*$ for a sufficiently large $n$.} 	
	
	\smallskip

	Let $t>0$ and, as $u_n$ belongs to $\mathcal{N}\setminus\mathcal{N}^*$ definitively in $n$ and $\gamma>1$, one yields to 
	$$ \int_{\Omega} f (u_n+t\varphi)^{1-\gamma} \le \int_{\Omega} f u_n^{1-\gamma} <\int_{\Omega} H^p(\nabla u_n)-\int_\Omega h u_n^{\theta+1}.$$
	Furthermore, requiring $t$ sufficiently small and thanks to the continuity of $\int_{\Omega} H^p(\nabla (u_n +t\varphi))$ (recall that $H$ is a norm) and of $\int_\Omega h (u_n + t\varphi)^{\theta+1}$ in $t$, from the the previous inequality one deduces   	
	$$ \int_{\Omega} f (u_n+t\varphi)^{1-\gamma} <\int_{\Omega} H^p(\nabla (u_n +t\varphi)) - \int_\Omega h(u_n+t\varphi)^{\theta+1},$$
	which means $u_n+t\varphi \in \mathcal{N}$ for $t$ small enough.

	Now we apply \eqref{eke2} with $w=u_n+t\varphi$ in order to get
	\begin{equation}\label{eke3}
		\begin{split}
		\frac 1p \left(\int_\Omega \left[H^p(u_n) - H^p(u_n+t\varphi)\right]\right) & - \frac{1}{\theta+1} \left(\int_{\Omega} h\left[u_n^{\theta+1}-(u_n+t\varphi)^{\theta+1}\right]\right)\\
		& +\frac{1}{\gamma-1} \left(\int_\Omega f \left[u_n^{1-\gamma} - (u_n+t\varphi)^{1-\gamma}\right]\right) \le \frac{\left(\displaystyle \int_{\Omega}|\nabla (t\varphi)|^p\right)^\frac{1}{p}}{n}.
		\end{split}
	\end{equation}
	Dividing by $t$ both sides of \eqref{eke3} and taking the liminf as $t$ tends to zero, recalling also that $u_n>0$ a.e. in $\Omega$, one yields to 
	\[
	\frac{\left(\displaystyle \int_{\Omega}|\nabla \varphi|^p\right)^\frac{1}{p}}{n} + \int_{\Omega} H^{p-1}(\nabla u_n) \nabla H(\nabla u_n) \cdot \nabla \varphi \ge \int_{\Omega}\frac{f\varphi}{u_n^\gamma}+\int_{\Omega} h u_n^{\theta} \varphi.
	\]
	Therefore, by taking the liminf as $n\to +\infty$ in the previous inequality, thanks to Proposition \ref{prop:convgrad} and to the Fatou Lemma, one obtains \eqref{disvar}.

	\medskip
	
	\emph{ii) $u_n$ belongs to $\mathcal{N}^*$ for a sufficiently large $n$.}
	\\ First observe that for any nonnegative $\varphi\in W^{1,p}_0(\Omega)$ and for any $t\ge 0$ one gets
	\begin{equation}\label{caso2_0}
		\int_{\Omega} f (u_n + t\varphi)^{1-\gamma} \le \int_{\Omega} f u_n^{1-\gamma}<+\infty,
	\end{equation}
	as $u_n\in \mathcal{N}^*$.
	Now let us observe that a priori $u_n+t\varphi$ may not belong to $\mathcal{N}^*$. Thus we notice that there exists $\lambda(t,n,\varphi)$ such that $\lambda(t,n,\varphi)(u_n + t\varphi) \in \mathcal{N}^*$. Moreover $\lambda$ satisfies the following
	\begin{equation}\label{qN*}
	\lambda^p(t,n,\varphi) \int_\Omega H^p(\nabla (u_n + t \varphi)) = \lambda^{1-\gamma}(t,n, \varphi) \int_\Omega f(u_n+t\varphi)^{1-\gamma} + \lambda^{\theta+1}(t,n,\varphi) \int_\Omega h(u_n+t\varphi)^{\theta+1}.
	\end{equation}
	We point out that $\lambda$ is not explicit in general. 
	Moreover, it follows from the Lebesgue Theorem and by continuity of the norm that $\lambda$ is continuous with respect to $t$. We also point out that $\lambda(0,n,\varphi)=1$ as, once again, $u_n \in \mathcal{N}^*$. \\
	We denote by 
	$$\lambda'(0,n,\varphi)= \lim_{t\to 0}\frac{\lambda(t,n,\varphi) - 1}{t},$$ 
	while, in case the previous limit does not exist,  $\lambda'(0,n,\varphi)$ is intended as $\displaystyle \lim_{k\to +\infty}\frac{\lambda(t_k,n,\varphi) - 1}{t_k}$ where the sequence $t_k \to 0^+$ is chosen in such a way that the limit exists.
	
	Now we aim to show that $|\lambda'(0,n,\varphi)|$ is bounded independently from $n$. To this end, we take the difference between \eqref{qN*} and the same formula evaluated in $t=0$ obtaining
	\begin{equation*}
	\begin{split}
		\lambda^p(t,n,\varphi)\int_{\Omega}H^p(\nabla (u_n + t\varphi)) - \int_{\Omega}H^p(\nabla u_n) =& \lambda^{1-\gamma}(t,n,\varphi)\int_{\Omega} f (u_n + t\varphi)^{1-\gamma} - \int_{\Omega} f u_n^{1-\gamma}\\
		&+ \lambda^{\theta+1}(t,n,\varphi) \int_\Omega h(u_n+t\varphi)^{\theta+1} - \int_\Omega h u_n^{\theta+1},
	\end{split}
	\end{equation*}
	which, after rearranging, it takes to 
	\begin{align*}
		0=&(\lambda^p(t,n,\varphi)-\lambda^p(0,n,\varphi))\int_{\Omega}H^p(\nabla (u_n + t\varphi)) +\int_{\Omega} \left[H^p(\nabla (u_n + t\varphi)) - H^p(\nabla u_n) \right]
		\\
		&- \lambda^{1-\gamma}(t,n,\varphi)\int_{\Omega} f(u_n + t\varphi)^{1-\gamma} + \int_{\Omega} f u_n^{1-\gamma} - \lambda^{\theta+1}(t,n,\varphi) \int_\Omega h(u_n+t\varphi)^{\theta+1} + \int_\Omega h u_n^{\theta+1}
		\\
		\ge& (\lambda^p(t,n,\varphi)-\lambda^p(0,n,\varphi))\int_{\Omega}H^p(\nabla (u_n + t\varphi)) + \int_{\Omega} \left[H^p(\nabla (u_n + t\varphi)) - H^p(\nabla u_n) \right]
		\\
		&- (\lambda^{1-\gamma}(t,n,\varphi)-\lambda^{1-\gamma}(0,n,\varphi))\int_{\Omega} f (u_n + t\varphi)^{1-\gamma} - (\lambda^{\theta+1}(t,n,\varphi) - \lambda^{\theta+1}(0,n,\varphi))\int_\Omega h(u_n+t\varphi)^{\theta+1} \\
		&- \int_\Omega h \left[(u_n+t\varphi)^{\theta+1} - u_n^{\theta+1}\right].
	\end{align*}
 Therefore, it follows from applications of the Mean Value Theorem that 
	\begin{equation}\label{caso2_1}
		\begin{aligned}
		0\ge & p \zeta_1^{p-1} (\lambda(t,n,\varphi)-1)\int_{\Omega}H^p(\nabla (u_n + t\varphi)) + \int_{\Omega} \left[H^p(\nabla (u_n + t\varphi)) - H^p(\nabla u_n) \right]
		\\
		&- (1-\gamma)\zeta_2^{-\gamma}(\lambda(t,n,\varphi)-1)\int_{\Omega} f (u_n + t\varphi)^{1-\gamma} - (\theta+1) \zeta_3^{\theta}  (\lambda(t,n,\varphi)-1)  \int_{\Omega} h(u_n+t\varphi)^{\theta+1}\\
		&- \int_\Omega h \left[(u_n+t\varphi)^{\theta+1} - u_n^{\theta+1}\right],
		\end{aligned}
	\end{equation}
	where $\zeta_1$, $\zeta_2$ and $\zeta_3$ are positive and tend to $1$ as $t \to 0^+$.
	Now let us divide by $t$ and then we pass to the limit as $t\to 0^+$ in \eqref{caso2_1}, yielding to
	\begin{align*}
		&0\ge p\lambda'(0,n,\varphi)\int_{\Omega}H^p(\nabla u_n) + p\int_{\Omega}H^{p-1}(\nabla u_n) \nabla H(\nabla u_n) \cdot \nabla\varphi + (\gamma-1) \lambda'(0,n,\varphi)\int_{\Omega} f u_n^{1-\gamma} \\
		&\qquad - (\theta+1) \lambda'(0,n,\varphi)  \int_{\Omega} h u_n^{\theta+1} - (\theta+1) \int_\Omega h u_n^{\theta} \varphi\\
		& \ \ =  p\int_{\Omega}H^{p-1}(\nabla u_n) \nabla H(\nabla u_n) \cdot \nabla\varphi - (\theta+1) \int_\Omega h u_n^{\theta} \varphi\\
		&\qquad + \lambda'(0,n,\varphi)\left(p\int_{\Omega}H^p(\nabla u_n)  + (\gamma-1)\int_{\Omega} f u_n^{1-\gamma} - (\theta+1) \int_\Omega h u_n^{\theta+1} \right)  
		\\
		&\stackrel{u_n\in \mathcal{N}^*}{=}  p\int_{\Omega}H^{p-1}(\nabla u_n) \nabla H(\nabla u_n) \cdot \nabla\varphi - (\theta+1) \int_\Omega h u_n^{\theta} \varphi \\
		&\qquad +  \lambda'(0,n,\varphi)\left((p+\gamma-1)\int_{\Omega} f u_n^{1-\gamma} + (p-1-\theta) \int_\Omega h u_n^{\theta+1} \right).
	\end{align*}
	Since $u_n$ is bounded in $W^{1,p}_0(\Omega)$ with respect to $n$ and, as we have shown, sequences into $\mathcal N$ stay away from some ball centered at $0$, then from the previous inequality we deduce that 
	$$\lambda'(0,n,\varphi)\le C,$$ 
	for some $C$ which does not depend on $n$.
	
	\medskip

	Now we show that $\lambda'(0,n,\varphi)$ is bounded from below. Indeed, fixing $w=\lambda(t,n,\varphi)(u_n +t\varphi)$ in \eqref{eke2}, one yields to 	
	$$
	\begin{aligned}
	&\frac{|\lambda(t,n,\varphi)-1|}{n}\left(\int_{\Omega} |\nabla u_n|^p\right)^\frac{1}{p} + \frac{t\lambda(t,n,\varphi)}{n}\left(\int_{\Omega} |\nabla \varphi|^p\right)^\frac{1}{p} 
	\ge\frac{1}{n}\left(\int_{\Omega} |\nabla (u_n-\lambda(t,n,\varphi)(u_n +t\varphi))|^p\right)^\frac{1}{p} \\
		\ge &J(u_n)-J(\lambda(t,n,\varphi)(u_n +t\varphi))
		= \frac{1}{p}\int_{\Omega}H^p(\nabla u_n) +\frac{1}{\gamma-1} \int_{\Omega} f u_n^{1-\gamma} - \frac{1}{\theta+1} \int_\Omega h u_n^{\theta+1}
		\\
		&- \frac{1}{p}\int_{\Omega}H^p(\nabla (u_n + t\varphi)\lambda(t,n,\varphi)) - \frac{1}{\gamma-1} \int_{\Omega} f \left[(u_n+t\varphi)\lambda(t,n,\varphi)\right]^{1-\gamma} \\
		&+ \frac{1}{\theta+1} \int_\Omega h\left[(u_n+t\varphi)\lambda(t,n,\varphi)\right]^{\theta+1}
		=\frac{1}{p}\int_{\Omega} \left[H^p(\nabla u_n)- H^p(\nabla (u_n + t\varphi) \lambda(t,n,\varphi)) \right]\\
		&+\frac{1}{\gamma-1} \int_{\Omega} f \left\{u_n^{1-\gamma}-\left[(u_n+t\varphi)\lambda(t,n,\varphi)\right]^{1-\gamma}\right\}
		- \frac{1}{\theta+1} \int_\Omega h\left\{u_n^{\theta+1}-\left[(u_n+t\varphi)\lambda(t,n,\varphi)\right]^{\theta+1}\right\}. 		
	\end{aligned}
	$$
	From the previous calculation,  recalling that $\lambda(t,n,\varphi)(u_n + t\varphi) \in \mathcal{N}^*$, one can estimate as
	\begin{align*}
		&\frac{t\lambda(t,n,\varphi)}{n}\left(\int_{\Omega} |\nabla \varphi|^p\right)^\frac{1}{p} + \left( \frac{1}{p} + \frac{1}{\gamma-1}\right)\left(\int_{\Omega}H^p(\nabla (u_n + t\varphi))  - \int_{\Omega}H^p(\nabla u_n)\right)\\
		&\ge -(\lambda^p(t,n,\varphi) - 1)\left(\frac{1}{p} + \frac{1}{\gamma-1}\right) \int_{\Omega}H^p(\nabla (u_n + t\varphi)) - \frac{|\lambda(t,n,\varphi)-1|}{n}\left(\int_{\Omega} |\nabla u_n|^p\right)^\frac{1}{p}\\
		& \quad  - \left(\frac{1}{\gamma-1}+\frac{1}{\theta+1}\right) \int_\Omega h\left\{u_n^{\theta+1}-\left[(u_n+t\varphi)\lambda(t,n,\varphi)\right]^{\theta+1}\right\}, \\
		&= -(\lambda^p(t,n,\varphi) - 1)\left(\frac{1}{p} + \frac{1}{\gamma-1}\right) \int_{\Omega}H^p(\nabla (u_n + t\varphi)) - \frac{|\lambda(t,n,\varphi)-1|}{n}\left(\int_{\Omega} |\nabla u_n|^p\right)^\frac{1}{p}\\
		& \quad  +(\lambda^{\theta+1}(t,n,\varphi) - 1) \left(\frac{1}{\gamma-1}+\frac{1}{\theta+1}\right) \int_\Omega h(u_n+t\varphi)^{\theta+1} + \left(\frac{1}{\gamma-1}+\frac{1}{\theta+1}\right) \int_\Omega h\left[(u_n+t\varphi)^{\theta+1}-u_n^{\theta+1}\right],
	\end{align*}
	from which, applying the Mean Value Theorem, one has
	\begin{equation}\label{caso2_2}
		\begin{aligned}
			&\frac{t\lambda(t,n,\varphi)}{n}\left(\int_{\Omega} |\nabla \varphi|^p\right)^\frac{1}{p} + \left( \frac{1}{p} + \frac{1}{\gamma-1}\right)\left(\int_{\Omega}H^p(\nabla (u_n + t\varphi))  - \int_{\Omega}H^p(\nabla u_n)\right)\\
			&\ge - p\zeta_4^{p-1}(\lambda(t,n,\varphi) - 1)\left(\frac{1}{p} + \frac{1}{\gamma-1}\right) \int_{\Omega}H^p(\nabla (u_n + t\varphi)) - \frac{|\lambda(t,n,\varphi)-1|}{n}\left(\int_{\Omega} |\nabla u_n|^p\right)^\frac{1}{p}\\
			& \quad   + (\theta+1)\zeta_5^{\theta}(\lambda(t,n,\varphi) - 1) \left(\frac{1}{\gamma-1}+\frac{1}{\theta+1}\right) \int_\Omega h(u_n+t\varphi)^{\theta+1}\\
			& \quad + \left(\frac{1}{\gamma-1}+\frac{1}{\theta+1}\right) \int_\Omega h\left[(u_n+t\varphi)^{\theta+1}-u_n^{\theta+1}\right],
		\end{aligned}
	\end{equation}
		where $\zeta_4, \zeta_5$ are positive and tend to $1$ as $t\to 0$. Then let us divide \eqref{caso2_2} by $t>0$ and pass to the limit as $t\to 0$, obtaining
	\begin{align*}
		&\frac{\displaystyle\left(\int_{\Omega} |\nabla \varphi|^p\right)^\frac{1}{p}}{n} + \left(1 + \frac{p}{\gamma-1}\right)\int_\Omega H^{p-1}(\nabla u_n)\nabla H(\nabla u_n)\cdot \nabla \varphi\\
		&\ge \lambda'(0,n,\varphi) \left[-\left( 1 + \frac{p}{\gamma-1}\right)\int_{\Omega}H^p(\nabla u_n)  
		- \frac{\operatorname{sgn}(\lambda'(0,n,\varphi))}{n}\left(\int_{\Omega} |\nabla u_n|^p\right)^\frac{1}{p} + \left(1+\frac{\theta+1}{\gamma-1}\right) \int_\Omega h u_n^{\theta+1}\right]\\
		&\quad + \left(1+\frac{\theta+1}{\gamma-1}\right) \int_\Omega h u_n^{\theta} \varphi,
	\end{align*}	
	which, getting rid of the nonnegative last term, implies that 
	\begin{align*}
		\lambda'(0,n,\varphi) & \left[\left( 1 + \frac{p}{\gamma-1}\right)\int_{\Omega}H^p(\nabla u_n)  
		+\frac{\operatorname{sgn}(\lambda'(0,n,\varphi))}{n}\left(\int_{\Omega} |\nabla u_n|^p\right)^\frac{1}{p} - \left(1+\frac{\theta+1}{\gamma-1}\right) \int_\Omega h u_n^{\theta+1}\right]
		\\
		&\ge- \frac{\displaystyle\left(\int_{\Omega} |\nabla \varphi|^p\right)^\frac{1}{p}}{n} - \left(1 + \frac{p}{\gamma-1}\right)\int_\Omega H^{p-1}(\nabla u_n)\nabla H(\nabla u_n)\cdot \nabla \varphi.
	\end{align*}	
	From the previous inequality one can deduce that $\lambda'(0,n,\varphi) \ge c> -\infty$ definitely in $n$ where $c$ is independent of $n$. Indeed, definitely in $n$ and since $u_n$ belongs to $\mathcal{N}^*$, one has 
	\begin{align*}
	&\left[\left( 1 + \frac{p}{\gamma-1}\right)\int_{\Omega}H^p(\nabla u_n)  
	+\frac{\operatorname{sgn}(\lambda'(0,n,\varphi))}{n}\left(\int_{\Omega} |\nabla u_n|^p\right)^\frac{1}{p} - \left(1+\frac{\theta+1}{\gamma-1}\right) \int_\Omega h u_n^{\theta+1}\right]\\
	&=\left[\left(\frac{p-1-\theta}{\gamma-1}\right)\int_{\Omega}H^p(\nabla u_n)  
	+\frac{\operatorname{sgn}(\lambda'(0,n,\varphi))}{n}\left(\int_{\Omega} |\nabla u_n|^p\right)^\frac{1}{p} +  \left(1+\frac{\theta+1}{\gamma-1}\right) \int_\Omega f u_n^{1-\gamma}\right]\ge C_1>0,
	\end{align*}
	with $C_1$ independent of $n$ as $\displaystyle 0<c_1\le \int_\Omega |\nabla u_n|^p\le c_2$ since we have already shown that $u_n$ is bounded in $W^{1,p}_0(\Omega)$ and it also stays away from zero in the same norm. Moreover, an application of the H\"older inequality gives that 
	$$\left|-\frac{\left(\int_{\Omega} |\nabla \varphi|^p\right)^\frac{1}{p}}{n} - \left(1 + \frac{p}{\gamma-1}\right)\int_\Omega H^{p-1}(\nabla u_n)\nabla H(\nabla u_n)\cdot \nabla \varphi\right| \le c,$$
	as it follows, once again, since $u_n$ is bounded in $W^{1,p}_0(\Omega)$ with respect to $n$.
	
	Then, we are ready to show the validity of \eqref{disvar} even in this case.
	Let $\varphi\in W^{1,p}_0(\Omega)$ be nonnegative and observe that it follows from \eqref{eke2} that
	\begin{align*}
		&\frac{|\lambda(t,n,\varphi)-1|}{n}\left(\int_{\Omega} |\nabla u_n|^p\right)^\frac{1}{p} + \frac{t\lambda(t,n,\varphi)}{n}\left(\int_{\Omega} |\nabla \varphi|^p\right)^\frac{1}{p} 
		\\
		&\ge \frac{1}{p}\int_{\Omega}H^p(\nabla u_n) +\frac{1}{\gamma-1} \int_{\Omega} f u_n^{1-\gamma} - \frac{1}{\theta+1} \int_\Omega h u_n^{\theta+1}
		\\
		&\quad - \frac{1}{p}\int_{\Omega}H^p(\nabla ((u_n + t\varphi)\lambda(t,n,\varphi))) - \frac{1}{\gamma-1} \int_{\Omega} f [(u_n+t\varphi)\lambda(t,n,\varphi)]^{1-\gamma} \\
		& \quad + \frac{1}{\theta+1} \int_\Omega h \left[(u_n+t\varphi)\lambda(t,n,\varphi)\right]^{\theta+1}
		\\
		&= - \frac{(\lambda(t,n,\varphi)^p-1)}{p}\int_{\Omega}H^p(\nabla (u_n + t\varphi)) 
		- \frac{1}{p} \left(\int_{\Omega}H^p(\nabla (u_n + t\varphi)) - \int_{\Omega}H^p(\nabla u_n)\right)
		\\
		&\quad - \frac{(\lambda(t,n,\varphi)^{1-\gamma}-1)}{\gamma-1}\int_{\Omega}f(u_n+t\varphi)^{1-\gamma} - \frac{1}{\gamma-1}\left(\int_{\Omega} f(u_n+t\varphi)^{1-\gamma} - \int_{\Omega} fu_n^{1-\gamma}\right)\\
		&\quad +\frac{(\lambda(t,n,\varphi)^{\theta+1}-1)}{\theta+1}\int_{\Omega}h(u_n+t\varphi)^{\theta+1} + \frac{1}{\theta+1}\left(\int_{\Omega} h(u_n+t\varphi)^{\theta+1} - \int_{\Omega} h u_n^{\theta+1}\right). 
	\end{align*}
	Then, applying again the Mean Value Theorem, one gets
	\begin{equation}\label{caso2_3}
	\begin{aligned}
		&\frac{|\lambda(t,n,\varphi)-1|}{n}\left(\int_{\Omega} |\nabla u_n|^p\right)^\frac{1}{p} + \frac{t\lambda(t,n,\varphi)}{n}\left(\int_{\Omega} |\nabla \varphi|^p\right)^\frac{1}{p} 
		\\
		&\ge - \zeta_1^{p-1}(\lambda(t,n,\varphi)-1)\int_{\Omega}H^p(\nabla (u_n + t\varphi)) 
		- \frac{1}{p} \left(\int_{\Omega}H^p(\nabla (u_n + t\varphi)) - \int_{\Omega}H^p(\nabla u_n)\right)
		\\
		&\quad + \zeta_2^{-\gamma} (\lambda(t,n,\varphi)-1)\int_{\Omega}f(u_n+t\varphi)^{1-\gamma} - \frac{1}{\gamma-1}\left(\int_{\Omega} f(u_n+t\varphi)^{1-\gamma} - \int_{\Omega} fu_n^{1-\gamma}\right)\\
		&\quad + \zeta_3^{\theta}(\lambda(t,n,\varphi)-1) \int_{\Omega}h(u_n+t\varphi)^{\theta+1} + \frac{1}{\theta+1}\left(\int_{\Omega} h(u_n+t\varphi)^{\theta+1} - \int_{\Omega} h u_n^{\theta+1}\right),
	\end{aligned}
	\end{equation}
	where $\zeta_1, \zeta_2$ and $\zeta_3$ are positive and tend to $1$ as $t\to 0^+$.
	Now let us divide \eqref{caso2_3} by $t$, passing to the limit as $t$ goes to $0$ and, taking into account that $u_n \in \mathcal{N}^*$ and \eqref{caso2_0}, one has
	\begin{gather} \label{limittcase2}
		\begin{split}
		&\frac{|\lambda'(0,n,\varphi)|}{n}\left(\int_{\Omega} |\nabla u_n|^p\right)^\frac{1}{p} + \frac{1}{n} \left( \displaystyle\int_{\Omega} |\nabla \varphi|^p\right)^\frac{1}{p} \\
		&\ge 	- \int_{\Omega}H^{p-1}(\nabla u_n) \nabla H(\nabla u_n)\cdot \nabla \varphi + \int_{\Omega} fu_n^{-\gamma}\varphi+\int_\Omega h u_n^\theta \varphi. 
		\end{split}
	\end{gather}
	Since  $|\lambda'(0,n,\varphi)|$ is bounded and $u_n$ is bounded in $W^{1,p}_0(\Omega)$ with respect to $n$, one can pass to the limit as  $n\to+\infty$ into \eqref{limittcase2} making use of Fatou Lemma, obtaining that
	\begin{gather*} 
		0 \ge - \int_{\Omega}H^{p-1}(\nabla u)\nabla H(\nabla u) \cdot \nabla \varphi + \int_{\Omega} fu^{-\gamma}\varphi + \int_\Omega h u^\theta \varphi.
	\end{gather*}		
	
	Therefore, regardless of the case considered, we have shown that \eqref{disvar} holds.
	
	\medskip
	
	To conclude the proof we are left to prove that $u$ is a weak solution to problem \eqref{PbMain}. 
%
Therefore let us take $\varphi= (u+\varepsilon\phi)^+$ in \eqref{disvar} for $\varepsilon>0$ and $\phi\in W^{1,p}_0(\Omega)$ and, as $u$ belongs to $\mathcal{N}^*$, one yields to
	\begin{align*}
		\displaystyle 0\le& \int_{\{u+\varepsilon\phi\ge 0\}} H^{p-1}(\nabla u)\nabla H(\nabla u) \cdot \nabla (u+\varepsilon\phi) - \int_{\{u+\varepsilon\phi\ge 0\}} \frac{f(u+\varepsilon\phi)}{u^\gamma} 	-  \int_{\{u+\varepsilon\phi\ge 0\}} h u^\theta (u+\varepsilon \phi)
		\\
		=& \int_{\Omega} H^{p-1}(\nabla u)\nabla H(\nabla u) \cdot \nabla (u+\varepsilon\phi) - \int_{\Omega} \frac{f(u+\varepsilon\phi)}{u^\gamma } -  \int_\Omega h u^\theta (u+\varepsilon \phi)\\
		& - \int_{\{u+\varepsilon\phi< 0\}} H^{p-1}(\nabla u)\nabla H(\nabla u) \cdot \nabla (u+\varepsilon\phi)+ \int_{\{u+\varepsilon\phi< 0\}} \frac{f(u+\varepsilon\phi)}{u^\gamma} +  \int_{\{u+\varepsilon\phi < 0\}} h u^\theta (u+\varepsilon \phi)	\\
		=&\int_{\Omega} H^{p-1}(\nabla u)\nabla H(\nabla u) \cdot \nabla u -  \int_{\Omega} fu^{1-\gamma} - \int_\Omega hu^{\theta+1}\\
		& + \varepsilon \left( \int_{\Omega} H^{p-1}(\nabla u)\nabla H(\nabla u) \cdot \nabla \phi - \int_{\Omega} \frac{f\phi}{u^\gamma} - \int_\Omega h u^\theta \phi\right) 
		\\
		&- \int_{\{u+\varepsilon\phi< 0\}} H^{p-1}(\nabla u)\nabla H(\nabla u) \cdot \nabla (u+\varepsilon\phi) + \int_{\{u+\varepsilon\phi< 0\}} \frac{f(u+\varepsilon\phi)}{u^\gamma} +  \int_{\{u+\varepsilon\phi < 0\}} h u^\theta (u+\varepsilon \phi)
		\\
		=& \int_{\Omega} H^{p}(\nabla u) -  \int_{\Omega} fu^{1-\gamma} - \int_\Omega h u^{\theta+1} \\
		&+ \varepsilon \left( \int_{\Omega} H^{p-1}(\nabla u)\nabla H(\nabla u) \cdot \nabla \phi - \int_{\Omega} \frac{f\phi}{u^\gamma} - \int_\Omega h u^\theta \phi - \int_{\{u+\varepsilon\phi< 0\}} H^{p-1}(\nabla u)\nabla H(\nabla u) \cdot \nabla \phi\right)
		\\
		&- \int_{\{u+\varepsilon\phi< 0\}} H^{p}(\nabla u) + \int_{\{u+\varepsilon\phi< 0\}} \frac{f(u+\varepsilon\phi)}{u^\gamma} +  \int_{\{u+\varepsilon\phi < 0\}} h u^\theta (u+\varepsilon \phi)
		\\
		\le&\varepsilon \left( \int_{\Omega} H^{p-1}(\nabla u)\nabla H(\nabla u) \cdot \nabla \phi -  \int_{\Omega} \frac{f\phi}{u^\gamma} - \int_\Omega h u^\theta \phi- \int_{\{u+\varepsilon\phi< 0\}} H^{p-1}(\nabla u)\nabla H(\nabla u) \cdot \nabla \phi\right),
	\end{align*}
	where for the last inequality we also used that $u \in \mathcal{N^*}$. 
	Let us now divide by $\varepsilon$ the previous inequality and let us take  $\varepsilon \to 0$, deducing that  
	\begin{equation*}
		\displaystyle \int_{\Omega} H^{p-1}(\nabla u)\nabla H(\nabla u) \cdot \nabla \phi - \int_{\Omega} \frac{f\phi}{u^\gamma} - \int_\Omega h u^\theta \phi \ge 0, \ \ \ \forall \phi \in W^{1,p}_0(\Omega).
	\end{equation*} 
	Taking $-\phi$ in place of $\phi$ we deduce the opposite inequality. This concludes the proof.
\end{proof}

\subsection{Proof of the existence results}

In this section we show existence Theorems \ref{intro:thm:existencegammaminoredi1}, \ref{intro:thm:existence1} and \ref{intro:thm:existence2}.
The first result we prove is Theorem  \ref{intro:thm:existencegammaminoredi1} whose proof is similar to the one of \cite[Theorem $3.2$]{DurOl} and for which we give only a sketch.

\begin{proof}[Proof of Theorem \ref{intro:thm:existencegammaminoredi1}]
	Let us consider an approximating sequence of solutions $u_n \in  W_0^{1,p}(\Omega)$ to 
	\begin{equation}\label{pbn}
		\begin{cases}
			\displaystyle -\Delta_{p}^{H}u_n = \frac{f_n}{(u_n+\frac{1}{n})^{\gamma}} + h_n T_n(u_n^{\theta})  & \text{in } \Omega,\\
			u_n = 0 & \text{on } \partial\Omega,
		\end{cases}
	\end{equation}
	where $f_n= T_n(f)$, $h_n=T_n(h)$. In particular, the existence of such $u_n$ classically follows from \cite{LL}. 
	Now let us test the weak formulation of \eqref{pbn} by $u_n$ itself and using the Euler Theorem, we get
	\[
	\int_{\Omega} H^p(\nabla u_n) \le \int_{\Omega} f_n u_n^{1-\gamma} + \int_{\Omega} h_n u_n^{1+\theta}.
	\]
	Now assume that $\gamma<1$; then it follows from \eqref{controlloH} and from the H\"older inequality that it holds
	\[
	\alpha^p \int_{\Omega} |\nabla u_n|^p \le \|f\|_{L^{(\frac{p^*}{1-\gamma})'}(\Omega)} \|u_n\|_{L^{p^*}(\Omega)}^{1-\gamma} + \|h\|_{L^{(\frac{p^*}{1+\theta})'}(\Omega)} \|u_n\|_{L^{p^*}(\Omega)}^{1+\theta}.
	\]
	Now an application of the Sobolev embedding theorem gives that $u_n$ is bounded in $W_0^{1,p}(\Omega)$ since both $1-\gamma$ and $1+\theta$ are strictly less than $p$. 
	In case $\gamma=1$ we repeat the same argument only for term involving $k$ as one take advantage of the fact $f_n\le f\in L^1(\Omega)$.
	
	Now the aim is passing to the limit with respect to $n$.	The passage to the limit is identical to the one given in the proof \cite[Theorem $3.2$]{DurOl}. The only observation is that, as $u_n$ is bounded in $W_0^{1,p}(\Omega)$, there exists a function $u \in W_0^{1,p}(\Omega)$ and a subsequence (still denoted by $u_n$) such that $u_n \rightharpoonup u$ weakly in $W_0^{1,p}(\Omega)$ and $u_n \to u$ strongly in $L^q(\Omega)$ for $q<p^*$ and almost everywhere.
	In particular, as the right hand of \eqref{pbn} is locally bounded in $L^1(\Omega)$ with respect to $n$, one can apply \cite[Theorem 2.1]{BoMu92} which provides that, up to a subsequence,
	\[
	\nabla u_n(x) \to \nabla u(x) \quad \text{for almost every } x \in \Omega.
	\]
	The above considerations allow to pass to the limit in left-hand of the weak formulation of \eqref{pbn} as already presented in \cite[Theorem $3.2$]{DurOl}. 
	This concludes the proof.
\end{proof}

Now, in the same spirit of \cite{lazer}, we are ready to prove the necessary and sufficient condition to have existence of weak solution to \eqref{PbMain}. The first theorem is for bounded functions $f$.

\begin{proof}[Proof of Theorem \ref{intro:thm:existence1}]
	
	We start by proving that a solution $u\in W^{1,p}_0(\Omega)$ to \eqref{PbMain} exists once $\gamma < 2 + \frac{1}{p-1}$.
	First observe that, if $\gamma \le 1$, this follows immediately from Theorem \ref{intro:thm:existencegammaminoredi1}; therefore, from here on, we assume $1<\gamma < 2 + \frac{1}{p-1}$. 
	In this case, to prove the existence  of a solution $u\in W^{1,p}_0(\Omega)$ to \eqref{PbMain}, we make use of Theorem \ref{thmexistence1}. 
	
	To this aim we consider $u_0 = \phi(x)^t$, for some $t>\frac{p-1}{p}$, where $\phi_1$ is the first eigenfunction of \eqref{einvaluprob}; let observe that, under the assumptions on $t$, $\phi(x)^t \in W^{1,p}_0(\Omega)$. 
	
	It holds
	\[
	\int_\Omega f|u_0|^{1-\gamma} \leq \|f\|_{L^\infty(\Omega)} \int_\Omega \phi_1^{t(1-\gamma)}. 
	\]	
	Therefore, as one can fix $t$ such that $t(1-\gamma)>-1$ if $\frac{p-1}{p}<\frac{1}{\gamma-1}$ and recalling \eqref{phi_integrabilita}, one has
	\begin{equation}\label{eq:eigenfuctconverg}
		\int_\Omega \phi_1^{t(1-\gamma)}  < +\infty \quad \text{if} \quad \gamma < 2 + \frac{1}{p-1}.
	\end{equation}
	Hence, since \eqref{eq:compatibility} is fulfilled, thanks to Theorem \ref{thmexistence1} we deduce that there exists a unique weak solution to \eqref{PbMain} for any $1<\gamma<2+\frac{1}{p-1}$.
	
	\medskip
		
	Now assume that $u\in W^{1,p}_0(\Omega)$ is a solution to \eqref{PbMain} in case $\gamma\ge 2+\frac{1}{p-1}$.
	Therefore there exists a non-negative sequence $w_n\in C^1_c(\Omega)$ converging to $u$ in $W^{1,p}_0(\Omega)$ as $n\to+\infty$. Then it follows from Theorem \ref{thm:estbelowabove}	and from the Fatou Lemma that one has 
	\[
	\liminf_{n\to+\infty}\int_{\Omega}w_nfu^{-\gamma}\ge \int_{\Omega}\frac{f}{u^{\gamma-1}} \geq \int_\Omega \frac{f}{a^{\gamma-1}\phi_1^{t(\gamma-1)}}=+\infty.
	\]
	On the other hand, we can take $w_n$ as a test function in the weak formulation of \eqref{PbMain} obtaining
	\[
	\displaystyle \int_{\Omega} H^{p-1}(\nabla u)\nabla u\cdot \nabla w_n= \int_{\Omega} \frac{f(x)}{u^{\gamma}}w_n+\int_\Omega h(x) u^\theta w_n,
	\]
	and passing to the limit as $n$ goes to +$\infty$, thanks to \eqref{eq:eulero} one gets
	\[
	\displaystyle\int_{\Omega} H^p(\nabla u)=+\infty,
	\]
	which contradicts that $u$ belongs to $W^{1,p}_0(\Omega)$. This concludes the proof.
\end{proof}

The second result is presented for unbounded $f$'s when $\gamma>1$.

\begin{proof}[Proof of Theorem \ref{intro:thm:existence2}]
	Assume first that $\gamma < 2 + \frac{1}{p-1} - \frac{p}{(p-1)m}$. Once again, we prove the existence  of a solution $u\in W^{1,p}_0(\Omega)$ to \eqref{PbMain} through an application of Theorem \ref{thmexistence1}. 
	
	Let $u_0 = \phi(x)^t$, for some $t>\frac{p-1}{p}$, where $\phi_1$ is the first eigenfunction of \eqref{einvaluprob}. Let recall that, under the assumptions on $t$, $\phi(x)^t \in W^{1,p}_0(\Omega)$. 
	
\begin{equation}
	\int_{\Omega} f u_0^{1-\gamma} \leq \left(\int_\Omega f^m\right)^{\frac{1}{m}} \cdot \left(\int_{\Omega} \phi_1^{t(1-\gamma)m'}\right)^{\frac{1}{m'}}.
\end{equation}
Even in this case one can fix $t$ such that	$t(1-\gamma)m'>-1$ if $\frac{p-1}{p}<\frac{m-1}{(\gamma-1)m}$. Then, once again by recalling \eqref{phi_integrabilita}, one has
\[
\int_{\Omega} \phi_1^{t(1-\gamma)m'} < +\infty \quad \text{if} \quad \gamma < 2 + \frac{1}{p-1} - \frac{p}{(p-1)m}.
\]
Therefore it follows from Theorem \ref{thmexistence1} that there exists a unique weak solution to \eqref{PbMain} for any $1<\gamma<2+\frac{1}{p-1}- \frac{p}{(p-1)m}$.

\medskip

Now observe that, in order to conclude the proof, it is sufficient to reason as in Theorem $5.7$ of \cite{OPsurvey} where the reverse implication is shown for $H(\xi) = |\xi|$, $p=2$ and $k=0$. Indeed the same reasoning holds even for a nonnegative and bounded function $k$. This concludes the proof.
\end{proof}

\subsection{Proof of Theorem \ref{thm:estbelowabove}}

In this section we show the behaviour of the weak solution in terms of the first eigenfunction of anisotropic $p$-laplacian, namely we prove Theorem \ref{thm:estbelowabove}.

\begin{proof}[Proof of Theorem \ref{thm:estbelowabove}]
	First observe that the equation solved by $u$ can be rewritten as
	\begin{equation}\label{pbsing}
		-\Delta_p^H u = \frac{g(x)}{u^\gamma} \quad \text{in } \Omega,
	\end{equation}
	where $g(x):=f(x)+h(x)u(x)^{\gamma+\theta}\in L^\infty(\Omega)$ as $u,f,h \in L^\infty(\Omega)$.  Moreover as $f$ is bounded from below in $\Omega$, one also gets that there exists $m,M>0$ such that $0 < m \le p \le M$ almost everywhere in $\Omega$.
	
	\medskip
	
	\textbf{Case $\gamma>1$.} 
	
	In this case the barriers we find for $u$ are of the form $\Psi := s\phi_1^\eta$ where $\eta={\frac{p}{\gamma+p-1}}$. Firstly, as we need them to be weak sub-supersolution to \eqref{pbsing}, they need to belong to $W_0^{1,p}(\Omega)$, namely that $|\nabla \Psi| = s\eta\phi_1^{\eta-1}|\nabla\phi_1| \in L^p(\Omega)$. In particular we have
	\[
	\int_\Omega |\nabla \Psi|^p = (st)^p \int_\Omega \phi_1^{p(\eta-1)}|\nabla\phi_1|^p
	\]
	and, as $|\nabla\phi_1|^p \in L^\infty(\Omega)$ and recalling \eqref{phi_integrabilita}, the integral on the right-hand of the previous is finite if and only if $p(\eta-1) > -1$, which is equivalent to require $\eta > \frac{p-1}{p}$. Thus one has
	\[
	\frac{p}{p-1+\gamma} > \frac{p-1}{p} \iff  \gamma < 2 + \frac{1}{p-1}.
	\]
	This shows that $\Psi$ belongs to $W_0^{1,p}(\Omega)$.
	
	\medskip
	Now we observe that
	\[
	-\Delta_p^H \Psi = \frac{q(x,s)}{\Psi^\gamma}, \quad \text{where} \quad q(x,s):=s^{\gamma+p-1} \eta^{p-1} \left[ (1-\eta)(p-1) H^p(\nabla \phi_1) + \lambda_1 \phi_1^p \right].
	\]
	As a supersolution we set $u_2=s_2\phi_1^\eta$ and we aim to show that 
	\begin{equation}\label{lmsuper}
		q(x,s_2) \ge g(x)
	\end{equation}
	for $s_2$ large enough. 
	
	Indeed, as $\phi_1 \in C^1(\overline{\Omega})$, it follows by the Hopf Lemma (see \cite[Theorem $4.5$]{CSR}) that one can fix $\varepsilon$ small enough such  that $|\nabla \phi_1| \not= 0$ on
	\begin{equation*}
		\Omega_{\varepsilon}=\{x\in\Omega : d(x)<\varepsilon\}.
	\end{equation*}
	If  $x\in \Omega_\varepsilon$ then \eqref{lmsuper} holds as one can ask for 
	\begin{equation}\label{lm1}
		s_2\ge \left(\frac{\displaystyle\max_{x\in\overline{\Omega}_\varepsilon} g(x)}{\eta^{p-1}(1-\eta)\alpha^p\displaystyle \min_{x\in \overline{\Omega}_\varepsilon}|\nabla\phi_1|^p}\right)^\frac{1}{\gamma+p-1}.
	\end{equation}
	Otherwise if $x\in \Omega\setminus\Omega_\varepsilon$ we need $s_2$ such that
	\begin{equation}\label{lm2}
		s_2\ge \left(\frac{\displaystyle\max_{x\in\Omega\setminus\Omega_\varepsilon} g(x)}{\eta^{p-1}\lambda_1\displaystyle \min_{x\in\Omega\setminus\Omega_\varepsilon}\phi_1^{p}}\right)^{\frac{1}{\gamma+p-1}}.
	\end{equation}
	By taking $s_2$ as the maximum among the right-hand of \eqref{lm1} and \eqref{lm2} then \eqref{lmsuper} holds. 
	
	We now look for a subsolution of the form $u_1 := s_1\phi_1^\eta$, where $s_1>0$. In particular we need
	\[
	q(x,s_1) \le m< g \ \text{in }\Omega.
	\]
	Then, as $\phi_1 \in C^1(\overline{\Omega})$ one has 
	\[
	\eta^{p-1} \left[ (1-\eta)(p-1) H^p(\nabla \phi_1) + \lambda_1 \phi_1^p \right]
	\]
	is a continuous function bounded from above in $\Omega$ by a positive constant $\overline{C} > 0$.
	
	Then we just need to show that 
	\[
	s_1^{\gamma+p-1} \overline{C} \le m,
	\]
	which is verified for $s_1 > 0$ small enough.

	\medskip
	
	\textbf{Case $0<\gamma\le 1$}.
	
	In this case we look for a supersolution of the type $u_2=s_2\phi_1^t$ ($s_2>0$) where  
	$t \in \left(\frac{p-1}{p}, 1\right)$. Analogously to the case $\gamma>1$ one has   
	\[
	-\Delta_p^H u_2 = \frac{q(x,s_2)}{u_2^\gamma}, \quad \text{where} \quad q(x,s_2):=s_2^{\gamma+p-1} t^{p-1} \left[ (1-t)(p-1) \phi_1^{t(p-1)-p+t\gamma}H^p(\nabla \phi_1) + \lambda_1 \phi_1^{t(p-1)+t\gamma} \right].
	\]
	Now observe that,  as $t(p-1)-p+t\gamma<0$, a very similar reasoning to the case $\gamma>1$, one can show that $q(x,s_2) \ge g(x)$ in $\Omega$ which holds fixing $s_2$ large enough. 
	
	For a subsolution we take $u_1 := s_1\phi_1$, where $s_1>0$. In particular we need
	\[
	s_1^{\gamma+p-1} \lambda_1 \phi_1^{p-1+\gamma} \le m< p \ \text{in }\Omega,
	\]
	which is verified for $s_1 > 0$ small enough. 
	
	\medskip
	
	Therefore, in order to conclude, one can apply Proposition \ref{thm:Comparison_Principle} which guarantees that $u_1 \le u \le u_2$ almost everywhere in $\Omega$. This concludes the proof.
\end{proof}

\end{document}